\documentclass[11pt]{article}

\usepackage{graphicx}%
\usepackage{multirow}%
\usepackage{amsmath,amssymb,amsfonts}%
\usepackage{amsthm}%
\usepackage{mathrsfs}%
\usepackage[title]{appendix}%
\usepackage{xcolor}%
\usepackage{textcomp}%
\usepackage{manyfoot}%
\usepackage{booktabs}%
\usepackage{algorithm}%
\usepackage{algorithmicx}%
\usepackage{algpseudocode}%
\usepackage{listings}%

\usepackage[hidelinks,bookmarksnumbered=true]{hyperref}
\usepackage{enumitem}   
\usepackage{cleveref}   
\usepackage{bm}         
\usepackage{bbm}        
\usepackage{bigints}    
\usepackage[square,numbers]{natbib}
\bibpunct[,]{(}{)}{,}{a}{}{,}%

\usepackage[symbol]{footmisc}
\usepackage{bookmark}   
\usepackage{authblk}    

\usepackage{caption}                
\usepackage{tikz}       
\usetikzlibrary{calc}
\usetikzlibrary{plotmarks}
\usetikzlibrary{backgrounds}
\usetikzlibrary{arrows.meta,positioning}
\usetikzlibrary{decorations.pathreplacing}
\usetikzlibrary{patterns}
\usepackage{xcolor}

\usepackage{pgfplots}
\pgfplotsset{compat=newest}
\pgfdeclareplotmark{mystar}{
	\node[star,star point ratio=2.25,minimum size=6pt,
	inner sep=0pt,draw=black,solid,fill=red] {};
}
\usepackage{subfigure}
\usepackage{graphicx}

\usepackage{multirow}
\usepackage{longtable}
\usepackage{tabularx}
\usepackage{makecell}
\usepackage{array}

\usepackage{thm-restate}    

\def\Re{\mathbb{R}}

\def\hat{\widehat}

\def\Re{{\mathbb R}}

\def\F{{\mathcal F}}

\def\N{{\mathcal N}}
\def\F{{\mathcal F}}
\def\T{{\mathcal T}}
\def\H{{\mathcal H}}

\def\L{{\mathcal L}}

\def\N{{\mathcal N}}

\def\P{{\mathcal P}}

\def\R{{\mathcal R}}
\def\T{{\mathcal T}}
\def\U{{\mathcal U}}

\def\X{{\mathcal X}}
\def\Y{{\mathcal Y}}

\def\e{{\mathbf e}}

\declaretheorem[name=Proposition]{proposition}

\declaretheorem[name=Corollary]{corollary}

\declaretheorem[name=Remark]{remark}

\declaretheorem[name=Example]{example}

\newcommand*{\QEDA}{\hfill\ensuremath{\square}}
\newcommand*{\QEDB}{\hfill\ensuremath{\diamond}}
\let\cline\cmidrule

\DeclareMathOperator*{\argmin}{arg\,min}

\def\rev{\textcolor{black}}
\definecolor{bluegreen}{RGB}{0,158,115}

\title{Regularized MIP Model for Integrating Energy Storage Systems and its Application for Solving a Trilevel Interdiction Problem}
\author[1]{Dahye Han\thanks{\href{mailto:dahye.han@gatech.edu}{dahye.han@gatech.edu}}}
\author[2]{Nan Jiang\thanks{\href{mailto:nanjiang@cornell.edu}{nanjiang@cornell.edu}}}
\author[1]{Santanu S. Dey\thanks{\href{mailto:santanu.dey@isye.gatech.edu}{santanu.dey@isye.gatech.edu}}}
\author[1]{Weijun Xie\thanks{\href{mailto:wxie@gatech.edu}{wxie@gatech.edu}}}

\affil[1]{School of Industrial and Systems Engineering, Georgia Institute of Technology}
\affil[2]{School of Operations Research and Information Engineering, Cornell Tech}

\date{}

\begin{document}
\maketitle	

\begin{abstract}
Incorporating energy storage systems (ESS) into power systems has been studied in many recent works, where binary variables are often introduced to model the complementary nature of battery charging and discharging. A conventional approach for these ESS optimization problems is to relax binary variables and convert the problem into a linear program. 
However, such linear programming relaxation models can yield unrealistic fractional solutions, such as simultaneous charging and discharging. In this paper, we develop a regularized Mixed-Integer Programming (MIP) model for the ESS optimal power flow (OPF) problem. We prove that under mild conditions, the proposed regularized model admits a zero integrality gap with its linear programming relaxation; hence, it can be solved efficiently. By studying the properties of the regularized MIP model, we show that its optimal solution is also near-optimal to the original ESS OPF problem, thereby providing a valid and tight upper bound for the ESS OPF problem. The use of the regularized MIP model allows us to solve a trilevel $\min$-$\max$-$\min$ network contingency problem which is otherwise intractable to solve.
\end{abstract}

\section{Introduction}
\label{sec:introduction}
Modern electrical grids have undergone significant transformations in the past few decades with increased integration of renewable energy resources and distributed energy resources. In spite of many benefits brought by these new entrants, power grids are also experiencing increased uncertainties due to inherent dependency on weather and short-term demand forecasts, which are challenging to accurately predict. To mitigate these challenges, many Independent System Operators (ISOs) are turning their attention to energy storage systems (ESS), also referred to as batteries for convenience in this paper. In their most recent annual study, the \cite{eia2023outlook} predicted 160 gigawatts of total installed battery storage capacity in the U.S. by the year 2050.

The introduction of ESS naturally leads us to revisit existing optimization problems in light of this new market entrant. Some previous studies include unit commitment problems with hydro storage (see, e.g., \citealt{jiang2012uc}), economic dispatch problems (see, e.g., \citealt{yan2016ed}), optimal bidding strategy for battery operators (see, e.g., \citealt{jiang2015optimal}), planning problems for wind farm and battery siting (see, e.g., \citealt{qi2015joint}), and control policy problems for optimizing the revenue of battery operators (see, e.g., \citealt{salas2018benchmarking}).

In this paper, we study the problem of incorporating ESS in the optimal power flow (OPF) problem. Incorporating batteries into the OPF problem adds a unique complexity since a battery can act as both a demand and a generator depending on whether it is charging or discharging. This introduces binary variables, making it a mixed-integer program (MIP) and difficult to solve planning problems of multi-level optimization. Here, we propose a new MIP model for the DC OPF problem with ESS that has a zero-integrality gap under certain conditions and we apply the model to a long-term planning problem to show the efficacy of the model.

\subsection{Relevant Literature}
Modeling the DCOPF problem with ESS has been of interest for more than a decade. 
\cite{xu2010opfbattery} proposed a simple OPF model incorporating ESS, demonstrating a pattern for the state-of-charge of a battery under the assumption of it being lossless for the case with a single generator and a single demand load.
However, this model oversimplifies the reality that batteries do incur losses.
\cite{ess_modelling} summarized several essential constraints for energy storage models including the energy loss factors with a round-trip efficiency of less than $1$. 
These constraints form the basis of many works that embed ESS in their corresponding optimization models (see the details in \citealt{lohndorf2023value,wu2023distributed}). 
Yet the model admits a computationally simple structure and fails to account for the complementary nature of charging and discharging.

A more accurate model of battery operations formulates the problem with binary variables to represent the complementary nature of charging and discharging for a battery. Although binary variables represent more realistic battery operations, mixed-integer models are more challenging to solve due to the nonconvex nature of the resulting formulations. To avoid this complexity, many papers use convex constraints to model the battery operations (see, e.g., \citealt{ess_uc,ess_uc_2,kody2022}). Specifically, \cite{pozo2022lpmodel} summarized various linear programming (LP) formulations that are valid relaxations of the MIP. However, these simple convex models may yield unrealistic solutions for the sake of simplicity. This issue has been discussed in \cite{ess_convex_model}, where counterexamples demonstrate that a battery charges and discharges simultaneously despite satisfying all conditions presented for strong convex relaxation models.

The scarcity of long-term planning studies that include ESS could be attributed to the lack of realistic yet efficient models for DCOPF with batteries, despite its potential to enhance reliability in the smart grid.
The $N-k$ contingency problem is one of these critical long-term planning problems that addresses disruptions or attacks within a system. 
Power grids face contingencies stemming from cyber-security attacks (see~\citealt{doe2021cybersecurity}) as well as unforeseen transmission line failures due to various physical threats, including an increasing frequency of wildfire incidents over the years.
While \citet{yang2022opfn-1} studied optimal power flow under $N-1$ contingency in power systems, to the best of our knowledge, no prior work has studied the battery placement problem under $N-k$ contingency.

\subsection{Summary of Contributions}
In this paper, we propose a regularized MIP model for battery operations in the DC optimal power flow problem.
The main benefits of the regularized MIP model are summarized below.
\begin{itemize}
    \item \textit{Regularized MIP model with zero integrality gap}. \quad In the exact battery model, one must enforce $p^c_t p^d_t = 0$, 
    where $p^c_t$ and $p^d_t$ are the amount of power being used to charge and discharge a battery at time $t$, respectively. 
    One can view this as enforcing a very specific sparsity condition. Often, sparsity is achieved by the addition of an $\ell_1$ regularizer penalty (see, e.g., \citealt{tibshirani1996regression,dey2022using}). In the same spirit, we perturb the original objective function of DCOPF with batteries by adding $\ell_1$ regularizers with respect to $p^c_t$ and $p^d_t$ for all times $t$. We prove that under mild conditions that are standard in most of the literature (see, e.g., \citealt{pozo2022lpmodel,kody2022}) and for a sufficiently large penalty, where the penalty value depends only on the efficiency of the battery (see details in Section~\ref{sec:regularized_mip_model}), the regularized MIP has zero integrality gap with its LP relaxation. 
    This regularized MIP model achieves the goal of being simple to solve yet produces feasible solutions for the actual battery operations. Moreover, the required penalty value is quite small for standard battery efficiencies, thus yielding near-optimal solutions in all our studies.
    \item \textit{High-quality upper bound}.\quad The optimal solution \rev{to} the regularized MIP model is a feasible solution to the original MIP and provides a valid and tight upper bound. We formally study the structural difference between the optimal solution \rev{to} the regularized problem and the original battery problem, provide an exactness condition, that is a condition under which we obtain the same solution, and prove a worst-case bound on the gap between their optimal objective values.
    We also evaluate our regularized MIP with a specific choice of regularizer parameter and empirically show that in practice the relative gap is small.
    \item \textit{Application to a long-term planning problem.} Leveraging the benefit of having no integrality gap for the regularized MIP, we examine the challenging application of trilevel $\min$-$\max$-$\min$ problem involving binary decision variables at each level. At the outermost level, a network designer is planning the locations of the batteries. The middle level is an interdictor allowed to attack a budgeted amount of the network. The innermost minimization problem is the system operator solving a DCOPF problem with batteries. 
    Using our regularized model in the third level allows us to replace an integer program with its LP relaxation. We can then take the dual of this LP relaxation, thus allowing the problem to be reduced to a bilevel problem with bounded dual variables (see details in \Cref{sec:trilevel_n-k_contingency}). The resulting bilevel problem can now be solved more efficiently with \rev{existing} combinatorial algorithms \rev{(see \citealt{bienstock2008computing} and \citealt{yang2022opfn-1})}.
    We empirically test on instances with up to $2000$ nodes and are able to solve this class of challenging trilevel instances in less than $6$ hours. To the best of our knowledge, this is the first time that the $N-k$ contingency problem for battery operations has been studied. This model may be of independent interest to the power systems research community.
\end{itemize}

\noindent\textbf{Organization.}
The remainder of the paper is organized as follows. Section~\ref{sec:original_formulation} provides a detailed mathematical formulation of the DCOPF with ESS. 
Section~\ref{sec:regularized_mip_model} introduces the regularized MIP model, presents its structural properties, and provides a comparison against the original MIP model. 
Section~\ref{sec:trilevel_n-k_contingency} introduces the trilevel $N-k$ contingency problem with ESS siting.
Section~\ref{sec:numerical_experiment} details the experimental settings and shows the computational results that demonstrate the power of using regularized formulation.
Section~\ref{sec:conclusion} concludes the paper.

\noindent\textbf{Notation.} The following notation is used throughout the paper. We use bold letters (e.g., $\bm{x},\bm{y}$) to denote vectors and matrices and use corresponding non-bold letters to denote their components. Given an integer $n$, we let $[n]:=\{1,2,\ldots,n\}$, and use $\Re_+^n:=\{\bm {x}\in \Re^n:x_i\geq0, \forall i\in [n]\}$. We let $\e$ be the vector or matrix of all ones and let $e_i$ be the $i$-th standard basis vector. Additional notations are introduced as needed. 

\section{OPF Formulation with Energy Storage Systems}
\label{sec:original_formulation}
In this section, we present the mathematical model for the DCOPF problem with batteries placed at a subset of network buses. 
Consider a network with a set of buses denoted as $\N$ and a set of transmission lines denoted as $\L$. 
The batteries are placed at a subset of buses $\N_b \subseteq \N$.
For simplicity, we assume that all batteries within the network have the same initial state-of-charge and configuration of efficiency level, lower and upper bounds on state-of-charge, charging rate, and discharging rate.  
Moreover, we assume these parameters remain unchanged over time. 
Our results on the regularized formulation in the next section do not require these assumptions. 
Let set $\T=\{1,\dots,T\}$ represent the finite time horizon with equal time intervals. 
At each time $t \in \T$, we decide the power output of each generator $p^g_{t,i}$ at bus $i \in \N$, taking into account the minimum and maximum generation limits, denoted by $G^{\max}_{i}$ and $G^{\min}_{i}$, respectively, i.e.,
\begin{subequations}
\begin{align}
    & G^{\min}_{i} \le p^g_{t,i} \le G^{\max}_{i}, &&  \forall  t \in \T, \; i \in \N.\label{eq:generation_limits}
\end{align}
Notice that if there is no generator at a certain bus $i \in \N$, we set $G_i^{\min}=G_i^{\max}=0$.
We also decide the power flow $f_{t,ij}$ through transmission line  $(i,\;j) \in \L$ subject to the limit for both directions $F_{ij}$, i.e., 
\begin{align}
    & -F_{ij} \le f_{t,ij} \le F_{ij}, && \forall  t \in \T, \; (i,j) \in \L. \label{eq:flow_limits}
\end{align}
The flow on a transmission line is proportional to the difference in phase angles of the corresponding buses:
\begin{align}
    & f_{t,ij} = B_{ij} (\theta_{t,i} - \theta_{t,j}), && \forall  t \in \T, \; (i,j) \in \L, \label{eq:dc_power_flow}
\end{align}
where $B_{ij}$ is the susceptance of the transmission lines $(i,\;j) \in \L$.

For batteries placed at certain nodes $i\in\N_b$, we determine the operations of the batteries. Since a battery can only charge or discharge at a given time point, we use a binary variable $u_{t,i} \in \{0,1\}$ to denote charging ($u_{t,i}=1$) and discharging ($u_{t,i}=0$) states. 
When a battery is charging, we decide the charging amount $p^c_{t,i}$ subject to its upper and lower limits, denoted as $E^{\max}_c$ and $E^{\min}_c$, respectively.  
Similarly, when the battery is discharging, we decide the discharging amount $p^d_{t,i}$ subject to its upper limit $E^{\max}_d$ and lower limit $E^{\min}_d$, i.e.,
\begin{align}
    & E_c^{\min} u_{t,i} \le p^{c}_{t,i} \le E_c^{\max} u_{t,i}, && \forall  t \in \T, \; i \in \N_b, \label{eq:battery_pc_limits}\\
    & E_d^{\min} (1-u_{t,i}) \le p^{d}_{t,i} \le E_d^{\max} (1-u_{t,i}), && \forall  t \in \T, \; i \in \N_b. \label{eq:battery_pd_limits}
\end{align}
\rev{We use $\eta_c \in (0,1]$ to denote the charging efficiency and $1/\eta_d \in [1, \infty)$ to represent the discharging efficiency, accounting for energy losses incurred during imperfect round-trip energy conversions, which may result from factors such as friction.}
The state-of-charge, represented by $p^s_{t,i}$, evolves based on the amount of battery charging and discharging.
We assume an initial state-of-charge to be $E_0$, i.e.,
\begin{align}
    & p^{s}_{t,i} = p^{s}_{t-1,i} + \rev{\eta_c} \cdot p^{c}_{t,i}- {1}/\rev{\eta_d} \cdot p^{d}_{t,i}, && \forall  t \in \T, \; i \in \N_b, \label{eq:battery_storage}\\
    & p^{s}_{0,i} = E_{0}, &&  \forall i \in \N_b. \label{eq:battery_initial_state}
\end{align}
\rev{Note that we have $\eta_c \leq 1 / \eta_d$, since otherwise, \eqref{eq:battery_storage} may result in arbitrage power.}

The state-of-charge of batteries is subject to upper and lower bounds for reliable operations, i.e.,
\begin{align}
    & E^{\min} \le p^{s}_{t,i} \le E^{\max}, && \forall  t \in \T, \; i \in \N_b. \label{eq:battery_ps_limits}
\end{align}
In case there are no batteries at certain nodes $i \in \N \setminus \N_b$, all battery-related variables are set to zero, i.e.,
\begin{align}
    & p^{s}_{t,i} = p^{s}_{0,i} = p^{c}_{t,i} = p^{d}_{t,i} = u_{t,i} = 0, && \forall  t \in \T, \; i \in \N \setminus \N_b.\label{eq:no_battery}
\end{align}
These operational decisions may lead to load shedding $p^{ls}_{t,i}$ if available power at bus $i \in \N$ is insufficient to meet the demand $D_{t,i}$.
Conversely, the system may experience excess power $p^{ex}_{t,i}$ if available power exceeds demand. \rev{Excess power can be interpreted as the counterpart of load shedding. 
At a specific node, load shedding occurs when the available power is insufficient to meet the demand, while excess power arises when there is more power than needed. This surplus power can threaten
grid stability and may result in curtailment (see e.g., \citealt{nrelextremeweather,novacheck2024weather}). With the increasing penetration of renewable energy, which is highly variable, managing excess power is as important as managing load shedding.
In recent years, California has seen a surge in solar power being a challenge to their power grid (\citealt{washingtonpost}).}
These variables act as slack variables and are always nonnegative:
\begin{align}
    & \bm p^{ls},\; \bm p^{ex} \ge \bm 0. \label{eq:battery_plspex_lb}
\end{align}
It is important to note that in an optimal solution, only one of load shedding or excess power can occur at a given time at each bus, ensuring that $p^{ls}_{t,i} p^{ex}_{t,i}=0$ for all $i \in \N$ and $t \in \T$. 

Altogether, these decisions must satisfy the power balance equation:
\begin{align}
    & \sum_{j \in \delta^{+}_{i}} f_{t,ij} - \sum_{j \in \delta^-_{i}} f_{t,ji} = p^{g}_{t,i} - D_{t,i} - p^{c}_{t,i} + p^{d}_{t,i} + p^{ls}_{t,i} - p^{ex}_{t,i}, && \forall  t \in \T, \; i \in \N, \label{eq:power_balance}
\end{align}
where $\delta^+_i = \{j \in \N : (i,j) \in \L\}$ and $\delta^-_i = \{j \in \N : (j,i) \in \L\}$. 
\rev{
The objective is to minimize the system cost including generator costs as well as load shedding and excess power over the entire network buses $i \in \N$ and the time horizon $t \in \T$. 
Since load shedding and excess power can be detrimental to the system, usually large constant $M$ is multiplied to estimate the cost. 
We normalize the cost of load shedding and excess power to $1$. We let $c^g(\cdot)$ denote the normalized generator cost function, which is usually linear or convex quadratic. 
}
\end{subequations}

For the rest of the paper, for simplicity, we use $\bm p = (\bm p^g, \bm p^s, \bm p^c, \bm p^d, \bm p^{ls}, \bm p^{ex})$, 
\rev{the system-wide cost $c(\bm p) = c^g(\bm p^{g}) + \sum_{t \in \mathcal{T}} \sum_{i \in \mathcal{N}} \left(p^{ls}_{t,i}+p^{ex}_{t,i} \right)$,}
\rev{$T = |\T|$,} $N = |\N|$, and $N_b = |\N_b|$.
Using these notations, we are now ready to introduce the DCOPF problem with the battery, that is, 
\begin{equation}
\tag{Battery}
    \rev{z^{\mathrm{ori}}} = \min_{\bm \theta, \bm f,\bm p, \bm u} \left\{ c(\bm p) \colon \eqref{eq:generation_limits}-\eqref{eq:power_balance},\;\bm u\in\{0,1\}^{T \times N} \right\}.\label{formulation:original_mip} 
\end{equation}
Without loss of generality, we assume that $0 \le F_{ij}$ for all $(i,j) \in \L$, $0 \le G^{\min}_i \le G^{\max}_i$ for each $i \in \N$, $0 \le E^{\min} \le E^{\max}$, $0 \le E^{\min}_c \le E^{\max}_c$, and $0 \le E^{\min}_d \le E^{\max}_d$.
Hence, it follows immediately $\bm p^{g},\; \bm p^{s},\; \bm p^{c},\; \bm p^{d} \ge \bm 0$. 

\section{Regularized MIP Model}
\label{sec:regularized_mip_model}
In this section, we first introduce a regularized MIP model and provide conditions such that this regularized MIP has the same optimal objective function value as its LP relaxation.
To simplify the notation, we employ the function $g(\cdot)$ to map $\bm p$ to a two-dimensional vector: $g(\bm p) = \left(\sum_{t \in \T} \sum_{i \in \N}p^{c}_{t,i},\;\sum_{t \in \T} \sum_{i \in \N}p^{d}_{t,i} \right)^\top $.
Now, we introduce a regularization function aimed at penalizing $\bm p^c$ and $\bm p^d$ with a given $\bm \lambda = (\lambda_c,\; \lambda_d)^\top \in \mathbb{R}^2_+$, that is, 
\begin{equation}
\tag{Reg-Battery}
   \rev{z^{\mathrm{reg}}(\bm \lambda)} = \min_{\bm \theta, \bm f,\bm p,\bm u} \left\{ c(\bm p) + \bm \lambda ^\top  g(\bm p) \colon \eqref{eq:generation_limits}- \eqref{eq:power_balance},\; \bm u \in \{0,1\}^{T\times N} \right\}.\label{formulation:regularization}
\end{equation}
\rev{
Note that $\lambda_c$ and $\lambda_d$ are the penalty terms with respect to the normalized cost of load shedding $\bm p^{ls}$ and the excess power $\bm p^{ex}$.
}
In \eqref{formulation:regularization} problem, the only binary decision is $\bm u \in \{0,1\}^{  T\times N}$. 
By relaxing this binary variable $\bm u$ to be continuous, we have the following convex relaxation for \eqref{formulation:regularization} problem:
\begin{equation}
\tag{LP-Reg-Battery}
   \rev{z_l^{\mathrm{reg}}(\bm \lambda)} = \min_{\bm \theta, \bm f,\bm p,\bm u} \left\{ c(\bm p) + \bm \lambda^\top g(\bm p) \colon \eqref{eq:generation_limits}- \eqref{eq:power_balance},\; \bm u \in [0,1]^{T\times N} \right\}.\label{formulation:regularization_lp}
\end{equation}
One of our main results in this section is to provide nontrivial sufficient conditions such that \eqref{formulation:regularization} problem and \eqref{formulation:regularization_lp} problem have the same optimal objective function value. 

\begin{restatable}{theorem}{theoremiplpcondition}\label{theorem:ip=lp_condition}
Suppose that $E_c^{\min} = E_d^{\min}=0$. If \rev{$\lambda_c + \eta_c \eta_d \lambda_d \geq 1-\eta_c \eta_d$}, then we have that $\rev{z^{\mathrm{reg}}(\bm \lambda)} =  \rev{z^{\mathrm{reg}}_{l}(\bm \lambda)}$.
\end{restatable}

\begin{proof}
\eqref{formulation:regularization_lp} problem is a relaxation of \eqref{formulation:regularization} problem, so it remains to show that an optimal solution \rev{to} \eqref{formulation:regularization_lp} problem is achieved with $u_{t,i}\in \{0,1\}$ for all $t \in \T$ and $i \in \N_b$, which is equivalent to showing that $p^c_{t,i}=0$ or $p^d_{t,i}=0$ for all $t \in \T$ and $i \in \N$. Let $(\hat{\bm \theta},\hat{\bm f},\hat{\bm p},\hat{\bm u})$ be an optimal solution \rev{to} \eqref{formulation:regularization_lp} problem. 
Suppose that $\hat{p}^{c}_{t^*,i^*} >0$ and $\hat{p}^{d}_{t^*,i^*} >0$ for some $t^*\in \T$ and $i^*\in \N$. 
We show that we can always find another feasible solution $(\tilde{\bm \theta},\tilde{\bm f},\tilde{\bm p},\tilde{\bm u})$ such that at most one of  $\tilde{p}^{c}_{t^*,i^*}$ and $\tilde{p}^{d}_{t^*,i^*}$ is positive for this given $t^*\in \T$ and $i^*\in \N$ and the corresponding objective value is at least as good as that of $(\hat{\bm \theta},\hat{\bm f},\hat{\bm p},\hat{\bm u})$. Such a solution can be constructed as follows:
\begin{subequations}
\begin{align}
    & \tilde{\bm \theta} = \hat{\bm \theta}, \; \tilde{\bm f} = \hat{\bm f}, \; \tilde{\bm p}^g = \hat{\bm p}^g, \; \tilde{\bm p}^s = \hat{\bm p}^s, && \notag \\
    & \rev{\tilde{p}^c_{t,i} = \max\{\hat{p}^{c}_{t,i} - \hat{p}^{d}_{t,i}/(\eta_c \eta_d) , 0\},} && \forall t \in \T, \; i \in \N, \label{eq:tilde_pc}\\
    & \rev{\tilde{p}^d_{t,i} = \max \{\hat{p}^{d}_{t,i} - \eta_c \eta_d  \hat{p}^{c}_{t,i}, 0\}, }&& \forall t \in \T, \; i \in \N, \label{eq:tilde_pd}\\
    & \tilde{p}^{ls}_{t,i} = \max\{-\hat{p}^c_{t,i} + \hat{p}^d_{t,i} + \hat{p}^{ls}_{t,i} - \hat{p}^{ex}_{t,i} + \tilde{p}^c_{t,i} - \tilde{p}^d_{t,i}, 0 \}, && \forall t \in \T, \; i \in \N,  \label{eq:tilde_pls}\\
    & \tilde{p}^{ex}_{t,i} = \max \{\hat{p}^c_{t,i} - \hat{p}^d_{t,i} - \hat{p}^{ls}_{t,i} + \hat{p}^{ex}_{t,i} - \tilde{p}^c_{t,i} + \tilde{p}^d_{t,i}, 0 \}, && \forall t \in \T, \; i \in \N, \label{eq:tilde_pex}\\
    & \tilde{u}_{t,i} = \mathbbm{1} \left\{ \tilde{p}^c_{t,i}>0\right\},&& \forall t \in \T, \; i \in \N. \label{eq:tilde_u}
\end{align}
\end{subequations}
Obviously, $(\tilde{\bm \theta},\tilde{\bm f},\tilde{\bm p},\tilde{\bm u})$ satisfies all the constraints of \eqref{formulation:regularization_lp} problem. In particular, for each $t \in \T, i \in \N$, the following equality is satisfied from the power balance equation \eqref{eq:power_balance}:
\begin{align}
    - \hat{p}^{c}_{t,i} + \hat{p}^{d}_{t,i} + \hat{p}^{ls}_{t,i} - \hat{p}^{ex}_{t,i} = - \tilde{p}^{c}_{t,i} + \tilde{p}^{d}_{t,i} + \tilde{p}^{ls}_{t,i} - \tilde{p}^{ex}_{t,i}. \label{eq:new_feasible_sol}
\end{align}
\rev{We also note that $\tilde{p}^c_{t,i} \tilde{p}^d_{t,i} = 0$. Suppose otherwise that $\tilde{p}^c_{t,i} > 0$ and $\tilde{p}^d_{t,i} > 0$. The first inequality implies that $\eta_c \eta_d  \hat{p}^{c}_{t,i} > \hat{p}^{d}_{t,i}$ whereas the second inequality implies the opposite, $\hat{p}^{d}_{t,i} > \eta_c \eta_d  \hat{p}^{c}_{t,i}$, a contradiction.}
\rev{Since the generator level is kept the same, that is $\bm{\hat p}^g=\bm{\tilde p}^g$, the generator cost in $c(\bm p)$ do not change, hence we will only compare the portion in $c(\bm p)$ that corresponds to the load shedding and the excess power.}
\rev{Now,} using the fact that $\tilde{p}^{ls}_{t,i}  \tilde{p}^{ex}_{t,i} = 0$ and $\tilde{p}^{c}_{t,i}  \tilde{p}^{d}_{t,i}=0$ for each $t \in \T, i \in \N$, it is sufficient to consider the following four cases.
\begin{enumerate}[label={(Case \arabic*)},leftmargin=*]
    \item When $\tilde{p}^{d}_{t,i}=\tilde{p}^{ls}_{t,i}=0$, the \rev{objective function of \eqref{formulation:regularization_lp} problem excluding the generator cost portion is:}
      \begin{align*}
        \tilde{p}^{ls}_{t,i} + \tilde{p}^{ex}_{t,i} + \lambda_c \tilde{p}^{c}_{t,i} + \lambda_d \tilde{p}^{d}_{t,i} 
            & = \tilde{p}^{ex}_{t,i} + \lambda_c \tilde{p}^{c}_{t,i}.\end{align*}
From condition \eqref{eq:new_feasible_sol}, we have $\tilde{p}^{ex}_{t,i}  = 	 \hat{p}^{c}_{t,i} - \hat{p}^{d}_{t,i} - \hat{p}^{ls}_{t,i} + \hat{p}^{ex}_{t,i} - \tilde{p}^{c}_{t,i} + \tilde{p}^{d}_{t,i} + \tilde{p}^{ls}_{t,i} $. Then, we substitute it for the objective value of \eqref{formulation:regularization_lp} problem, that is,
\begin{align*}
    \tilde{p}^{ex}_{t,i} + \lambda_c \tilde{p}^{c}_{t,i} = \hat{p}^c_{t,i} - \hat{p}^d_{t,i} - \hat{p}^{ls}_{t,i} + \hat{p}^{ex}_{t,i} - \tilde{p}^c_{t,i} + \lambda_c \tilde{p}^{c}_{t,i}.
\end{align*}
Based on the construction in \eqref{eq:tilde_pc} and \eqref{eq:tilde_pd}, together with the presumption $\tilde{p}^{d}_{t,i}=0$, 
\rev{we obtain the following simplification:}
\begin{align*}
   \hat{p}^c_{t,i} - \hat{p}^d_{t,i} - \hat{p}^{ls}_{t,i} + \hat{p}^{ex}_{t,i} - \tilde{p}^c_{t,i} + \lambda_c \tilde{p}^{c}_{t,i} & = \hat{p}^c_{t,i} - \hat{p}^d_{t,i} - \hat{p}^{ls}_{t,i} + \hat{p}^{ex}_{t,i} + (\lambda_c - 1) \left( \hat{p}^{c}_{t,i} - \rev{\frac{\hat{p}^{d}_{t,i}}{\eta_c\eta_d}} \right)\\
     & = - \hat{p}^{ls}_{t,i} + \hat{p}^{ex}_{t,i} + \lambda_c \hat{p}^{c}_{t,i} + \rev{\frac{1-\lambda_c-\eta_c\eta_d}{\eta_c \eta_d}} \hat{p}^{d}_{t,i} \\
     & \leq \hat{p}^{ls}_{t,i} + \hat{p}^{ex}_{t,i} + \lambda_c \hat{p}^{c}_{t,i} + \lambda_d \hat{p}^{d}_{t,i}.
\end{align*}
\rev{The last inequality follows from the assumption $\lambda_c + \eta_c \eta_d \lambda_d \geq 1-\eta_c \eta_d$ along with the fact that $ \eta_c\eta_d > 0$.}
    \item When $\tilde{p}^{d}_{t,i}=\tilde{p}^{ex}_{t,i}=0$, the \rev{objective function of \eqref{formulation:regularization_lp} problem excluding the generator cost portion:}
\begin{align*}
   \tilde{p}^{ls}_{t,i} + \tilde{p}^{ex}_{t,i} + \lambda_c \tilde{p}^{c}_{t,i} + \lambda_d \tilde{p}^{d}_{t,i} & =\tilde{p}^{ls}_{t,i} + \lambda_c \tilde{p}^{c}_{t,i}. 
\end{align*}
From condition \eqref{eq:new_feasible_sol}, we have
\begin{align*}
  \tilde{p}^{ls}_{t,i} + \lambda_c \tilde{p}^{c}_{t,i}  =   - \hat{p}^c_{t,i} + \hat{p}^d_{t,i} + \hat{p}^{ls}_{t,i} - \hat{p}^{ex}_{t,i} + \tilde{p}^c_{t,i} + \lambda_c \tilde{p}^{c}_{t,i}.
\end{align*}
Similarly, based on the construction in \eqref{eq:tilde_pc} and \eqref{eq:tilde_pd}, together with the presumption $\tilde{p}^{d}_{t,i}=0$, \rev{we obtain the following simplification:}
\begin{align*}
  - \hat{p}^c_{t,i} + \hat{p}^d_{t,i} + \hat{p}^{ls}_{t,i} - \hat{p}^{ex}_{t,i} + \tilde{p}^c_{t,i} + \lambda_c \tilde{p}^{c}_{t,i} &  = - \hat{p}^c_{t,i} + \hat{p}^d_{t,i} + \hat{p}^{ls}_{t,i} - \hat{p}^{ex}_{t,i} + (\lambda_c + 1) \left( \hat{p}^{c}_{t,i} - \rev{\frac{\hat{p}^{d}_{t,i}}{\eta_c \eta_d}} \right)\\
  & = \hat{p}^{ls}_{t,i} - \hat{p}^{ex}_{t,i} + \lambda_c \hat{p}^{c}_{t,i} + \left(1- \rev{\frac{1}{\eta_c \eta_d} - \frac{\lambda_c}{\eta_c \eta_d}}\right) \hat{p}^{d}_{t,i} \\
  & \rev{\leq \hat{p}^{ls}_{t,i} + \hat{p}^{ex}_{t,i} + \lambda_c \hat{p}^{c}_{t,i} + \lambda_d \hat{p}^{d}_{t,i}.}
\end{align*}
\rev{
The last inequality follows from the assumption that  $\lambda_c + \eta_c \eta_d \lambda_d \geq 1-\eta_c \eta_d$ which implies $\lambda_d \geq -1 + \frac{1}{\eta_c \eta_d} - \frac{\lambda_c}{\eta_c \eta_d}$ with $\eta_c \eta_d > 0$ and $ \eta_c \leq 1 / \eta_d$ which implies $-1 + \frac{1}{\eta_c \eta_d} \geq 0 \geq 1 -  \frac{1}{\eta_c \eta_d}$.
}
    \item When $\tilde{p}^{c}_{t,i}=\tilde{p}^{ls}_{t,i}=0$, the \rev{objective function of \eqref{formulation:regularization_lp} problem excluding the generator cost portion is:}
    \begin{align*}
       \tilde{p}^{ls}_{t,i} + \tilde{p}^{ex}_{t,i} + \lambda_c \tilde{p}^{c}_{t,i} + \lambda_d \tilde{p}^{d}_{t,i} & =\tilde{p}^{ex}_{t,i} + \lambda_d \tilde{p}^{d}_{t,i}. 
    \end{align*}
From condition \eqref{eq:new_feasible_sol}, we have
\begin{align*}
 \tilde{p}^{ex}_{t,i} + \lambda_d \tilde{p}^{d}_{t,i} 
 =    \hat{p}^c_{t,i} - \hat{p}^d_{t,i} - \hat{p}^{ls}_{t,i} + \hat{p}^{ex}_{t,i} + \tilde{p}^d_{t,i} + \lambda_d \tilde{p}^{d}_{t,i}.
\end{align*}
Based on the construction in \eqref{eq:tilde_pc} and \eqref{eq:tilde_pd}, together with the presumption $\tilde{p}^{c}_{t,i}=0$, \rev{we obtain the following simplification:}
    \begin{align*}
    \hat{p}^c_{t,i} - \hat{p}^d_{t,i} - \hat{p}^{ls}_{t,i} + \hat{p}^{ex}_{t,i} + \tilde{p}^d_{t,i} + \lambda_d \tilde{p}^{d}_{t,i} 
    & = \hat{p}^c_{t,i} - \hat{p}^d_{t,i} - \hat{p}^{ls}_{t,i} + \hat{p}^{ex}_{t,i} + (\lambda_d + 1) (\hat{p}^{d}_{t,i} - \rev{\eta_c \eta_d}  \hat{p}^{c}_{t,i}) \\
     & = - \hat{p}^{ls}_{t,i} + \hat{p}^{ex}_{t,i} + \left(1 - \rev{\eta_c \eta_d} \lambda_d - \rev{\eta_c \eta_d}\right) \hat{p}^c_{t,i} + \lambda_d \hat{p}^{d}_{t,i} \\
     & \rev{ \leq   \hat{p}^{ls}_{t,i} + \hat{p}^{ex}_{t,i} + \lambda_c \hat{p}^{c}_{t,i} + \lambda_d \hat{p}^{d}_{t,i}.}
    \end{align*}
\rev{
The last inequality follows immediately from the assumption that  $\lambda_c + \eta_c \eta_d \lambda_d \geq 1-\eta_c \eta_d$.
}
    \item When $\tilde{p}^{c}_{t,i}=\tilde{p}^{ex}_{t,i}=0$, the \rev{objective function of \eqref{formulation:regularization_lp} problem excluding the generator cost portion is:} 
    \begin{align*}
    \tilde{p}^{ls}_{t,i} + \tilde{p}^{ex}_{t,i} + \lambda_c \tilde{p}^{c}_{t,i} + \lambda_d \tilde{p}^{d}_{t,i} & =\tilde{p}^{ls}_{t,i} + \lambda_d \tilde{p}^{d}_{t,i}.
    \end{align*}
From condition \eqref{eq:new_feasible_sol}, we have
\begin{align*}
  \tilde{p}^{ls}_{t,i} + \lambda_d \tilde{p}^{d}_{t,i} =  - \hat{p}^c_{t,i} + \hat{p}^d_{t,i} + \hat{p}^{ls}_{t,i} - \hat{p}^{ex}_{t,i} - \tilde{p}^d_{t,i} + \lambda_d \tilde{p}^{d}_{t,i}.
\end{align*}
Similarly, based on the construction in \eqref{eq:tilde_pc} and \eqref{eq:tilde_pd}, together with the presumption $\tilde{p}^{c}_{t,i}=0$, \rev{we obtain the following simplification:}
\begin{align*}
- \hat{p}^c_{t,i} + \hat{p}^d_{t,i} + \hat{p}^{ls}_{t,i} - \hat{p}^{ex}_{t,i} - \tilde{p}^d_{t,i} + \lambda_d \tilde{p}^{d}_{t,i} & = - \hat{p}^c_{t,i} + \hat{p}^d_{t,i} + \hat{p}^{ls}_{t,i} - \hat{p}^{ex}_{t,i} + (\lambda_d - 1) (\hat{p}^{d}_{t,i} - \rev{\eta_c \eta_d} \hat{p}^{c}_{t,i})\\
 & = \hat{p}^{ls}_{t,i} - \hat{p}^{ex}_{t,i} + \left(- \rev{\eta_c \eta_d} \lambda_d + \rev{\eta_c \eta_d} - 1 \right) \hat{p}^c_{t,i} +  \lambda_d \hat{p}^{d}_{t,i} \\
 & \rev{\leq  \hat{p}^{ls}_{t,i} + \hat{p}^{ex}_{t,i} + \lambda_c \hat{p}^{c}_{t,i} + \lambda_d \hat{p}^{d}_{t,i}. }
\end{align*}
\rev{
The last inequality follows from the assumption that  $\lambda_c + \eta_c \eta_d \lambda_d \geq 1-\eta_c \eta_d$ which implies $\lambda_c \geq - \eta_c \eta_d \lambda_d + 1-\eta_c \eta_d$ and $ \eta_c \leq 1 / \eta_d$ which implies $1-\eta_c \eta_d \geq 0 \geq \eta_c \eta_d - 1$.
}
\end{enumerate}
Since above four cases hold for each $t \in \T$ and $i \in \N$, we have 
\rev{
\begin{align*}
    c(\tilde{\bm p}) + \bm \lambda^\top g(\tilde{\bm p}) & = c^g(\bm{\tilde  p^{g}}) + \sum_{t \in \mathcal{T}} \sum_{i \in \mathcal{N}} \left(\tilde{p}^{ls}_{t,i} + \tilde{p}^{ex}_{t,i}\right) + \lambda_c \sum_{t \in \mathcal{T}} \sum_{i \in \mathcal{N}} \tilde{p}^{c}_{t,i} + \lambda_d \sum_{t \in \mathcal{T}} \sum_{i \in \mathcal{N}} \tilde{p}^{d}_{t,i}  \\
    & \le c^g(\bm{\hat p^{g}}) + \sum_{t \in \mathcal{T}} \sum_{i \in \mathcal{N}} \left(\hat{p}^{ls}_{t,i} + \hat{p}^{ex}_{t,i}\right) + \lambda_c \sum_{t \in \mathcal{T}} \sum_{i \in \mathcal{N}} \hat{p}^{c}_{t,i} + \lambda_d \sum_{t \in \mathcal{T}} \sum_{i \in \mathcal{N}} \hat{p}^{d}_{t,i}  \\
    & = c(\hat{\bm p}) + \bm \lambda^\top g(\hat{\bm p}).
\end{align*}
}
This completes the proof.
\QEDA
\end{proof}

\Cref{theorem:ip=lp_condition} demonstrates the equivalence between the optimal objective function value of \eqref{formulation:regularization} problem and \eqref{formulation:regularization_lp} problem under regularization (for sufficiently high penalty) and the assumption that $E^{\min}_c=E^{\min}_d=0$. Note that this assumption is standard and appears in many recent works, such as \citet{pozo2022lpmodel} and \citet{kody2022}. The technique to devise an integral solution by perturbing the charge and discharge levels has been used in the context of optimizing for a single solar-battery storage system in \cite{singh2021lagrangian}. 

Theorem~\ref{theorem:ip=lp_condition} easily leads to the following Corollary.
\begin{corollary}
\label{corollary:fori=freg}
    When \rev{$\eta_c=\eta_d=1$}
    for any $\bm \lambda \ge \bm 0$, \eqref{formulation:regularization} problem and \eqref{formulation:regularization_lp} problem are equivalent. In particular, when $\bm \lambda = \bm 0$, \eqref{formulation:original_mip} problem, \eqref{formulation:regularization} problem, and \eqref{formulation:regularization_lp} problem are all equivalent, i.e., $ \rev{z^{\mathrm{ori}}} = \rev{z^{\mathrm{reg}}(\bm 0)}=\rev{z_l^{\mathrm{reg}}(\bm 0)}$.
\end{corollary}
\rev{Therefore, when efficiencies for both charging and discharging are 1 (i.e., the battery is lossless and $\eta_c=\eta_d=1$), relaxing the integrality of battery operations is exact.}
However, as \rev{such} does not occur in practice, most literature concerning battery operation bases numerical experiments with \rev{$\eta_c, \eta_d < 1$.}

Finally, we remark that the two assumptions of \Cref{theorem:ip=lp_condition}, (i.e., $E_c^{\min} = E_d^{\min}=0$ and 
\rev{$\lambda_c + \eta_c \eta_d \lambda_d \geq 1-\eta_c \eta_d$}
 are the best that we may expect for the equivalence of the optimal objective function value of the regularized MIP model and its LP relaxation. The following two examples illustrate that \eqref{formulation:regularization} problem and \eqref{formulation:regularization_lp} problem do not have the same optimal objective function value if either of these two assumptions in \Cref{theorem:ip=lp_condition} is violated.

\begin{example}\rm
\label{ex:E_min_not_zero}
\textit{($\rev{z^{\mathrm{reg}}(\bm \lambda)} \neq \rev{z_l^{\mathrm{reg}}(\bm \lambda)}$ when $E^{\min}_c,\;E^{\max} > 0$)} Consider a simple network with $\N=\{1,2\}$, $\T=\{1,2\}$, $\L=\{(1,2)\}$. 
Suppose one battery is placed at node $2$ (there is no battery placed at node $1$) with $\N_b=\{2\}$ and $E_c^{\min} = E_d^{\min}=\tau$, $E_c^{\max} = E_d^{\max}=2$, $E^{\min}=0$, $E^{\max}=4$, $E_0=0$, and 
\rev{$\eta_c = \eta_d =1/2$. }
Assume each node has one generator with $G_1^{\min}=G_2^{\min} =2$ and $G_1^{\max}=G_2^{\max} = 4$
\rev{and generator cost $c^g(\bm p^g) = 0$}. 
We further assume $-4\leq f_{12} \leq 4$ and the demand is $D_{1,1}=2,\;D_{1,2}=4,\;D_{2,1}=6,\;D_{2,2}=4$. 
Without loss of generality, we assume that the Ohm's law constraint \eqref{eq:dc_power_flow} is satisfied.
When $\bm \lambda=(3/5,\;3/5)^\top$, for any $\tau\in(0,1/2]$, an optimal solution \rev{to} \eqref{formulation:regularization} problem is \begin{align*}
    & \hat{p}^c_{1,1} =0,\;\hat{p}^c_{1,2} =4 \tau,\;\hat{p}^c_{2,1}=0,\;\hat{p}^c_{2,2}=0, \\
    & \hat{p}^d_{1,1} =0,\;\hat{p}^d_{1,2} =0,\;\hat{p}^d_{2,1}=0,\;\hat{p}^d_{2,2}=\tau, \\
       & \hat{p}^{ls}_{1,1} =0,\;\hat{p}^{ls}_{1,2} =0,\;\hat{p}^{ls}_{2,1}=0,\;\hat{p}^{ls}_{2,2}=2-\tau, \\
            & \hat{p}^{ex}_{1,1} =0,\;\hat{p}^{ex}_{1,2} =0,\;\hat{p}^{ex}_{2,1}=0,\;\hat{p}^{ex}_{2,2}=0, \\  
    & \hat{u}_{1,1} =0,\;\hat{u}_{1,2}=1,\;\hat{u}_{2,1}=0,\;\hat{u}_{2,2}=0, 
\end{align*}
with the optimal objective value $\hat{v}=2+2\tau$.

While an optimal solution \rev{to} the corresponding \eqref{formulation:regularization_lp} problem is 
\begin{align*}
    & \tilde{p}^c_{1,1} =0,\;\tilde{p}^c_{1,2}=\tau,\;\tilde{p}^c_{2,1}=0,\;\tilde{p}^c_{2,2}=\frac{3\tau}{5}, \\ 
    & \tilde{p}^d_{1,1} =0,\;\tilde{p}^d_{1,2}=0,\;\tilde{p}^d_{2,1}=0,\;\tilde{p}^d_{2,2}=\frac{2\tau}{5}, \\
    & \tilde{p}^{ls}_{1,1} =0,\;\tilde{p}^{ls}_{1,2} =0,\;\tilde{p}^{ls}_{2,1}=0,\;\tilde{p}^{ls}_{2,2}=2+\frac{\tau}{5}, \\
    & \tilde{p}^{ex}_{1,1} =0,\;\tilde{p}^{ex}_{1,2} =0,\;\tilde{p}^{ex}_{2,1}=0,\;\tilde{p}^{ex}_{2,2}=0, \\  
 &\tilde{u}_{1,1}=0, \tilde{u}_{1,2}=1, \tilde{u}_{2,1}=0, \tilde{u}_{2,2}=\frac{3}{5}, 
\end{align*}
with the optimal objective value $\tilde{v}=2+7\tau/5$.  Therefore, $\tilde{v} <  \hat{v}$ for all $\tau\in (0,1/2]$. Hence, \eqref{formulation:regularization} problem and \eqref{formulation:regularization_lp} problem are not equivalent.
    \QEDB
\end{example}

\begin{example}\rm
\label{ex:ambda_less_than}
\textit{($\rev{z^{\mathrm{reg}}(\bm \lambda)} \neq \rev{z_l^{\mathrm{reg}}(\bm \lambda)}$ when 
\rev{$\lambda_c + \eta_c \eta_d \lambda_d < 1-\eta_c \eta_d$}
)} Consider the same network as that in Example \ref{ex:E_min_not_zero} but with different battery configurations and demands. One battery is placed at node $2$ (i.e., $\N_b=\{2\}$) with $E_c^{\min} = E_d^{\min}=0, E_c^{\max} = E_d^{\max}=1$, $E^{\min}=0$, $E^{\max}=6$, $E_0=6$, and \rev{$\eta_c^2=\eta_d^2=1/3$}. 
Demand is $D_{1,1}=1,\;D_{1,2}=2,\;D_{2,1}=4,\;D_{2,2}=8$.
Other parameters remain the same.
We assume that the Ohm's law constraint \eqref{eq:dc_power_flow} is satisfied.
Let $\bm \lambda=(\tau,\;\tau)^\top$ for any $\tau\in[0,1/2)$. An optimal solution \rev{to} \eqref{formulation:regularization} problem is 
\begin{align*}
    & \hat{p}^c_{1,1} =0,\;\hat{p}^c_{1,2} =0,\;\hat{p}^c_{2,1}=0,\;\hat{p}^c_{2,2}=0, \\
    & \hat{p}^d_{1,1} =0,\;\hat{p}^d_{1,2} =0,\;\hat{p}^d_{2,1}=0,\;\hat{p}^d_{2,2}=1, \\
    & \hat{p}^{ls}_{1,1} =0,\;\hat{p}^{ls}_{1,2} =0,\;\hat{p}^{ls}_{2,1}=0,\;\hat{p}^{ls}_{2,2}=3, \\
    & \hat{p}^{ex}_{1,1} =1,\;\hat{p}^{ex}_{1,2} =0,\;\hat{p}^{ex}_{2,1}=0,\;\hat{p}^{ex}_{2,2}=0, \\  
    & \hat{u}_{1,1} =0, \; \hat{u}_{1,2}=0, \; \hat{u}_{2,1}=0, \; \hat{u}_{2,2}=0, 
\end{align*}
with the optimal objective value $\hat{v}=4+\tau$.

While an optimal solution \rev{to} the corresponding \eqref{formulation:regularization_lp} problem is 
\begin{align*}
 & \tilde{p}^c_{1,1} =0,\tilde{p}^c_{1,2}=3/4,\tilde{p}^c_{2,1}=0,\tilde{p}^c_{2,2}=0, \\ & \tilde{p}^d_{1,1} =0,\tilde{p}^d_{1,2}=1/4,\tilde{p}^d_{2,1}=0,\tilde{p}^d_{2,2}=1, \\
    & \tilde{p}^{ls}_{1,1} =0,\tilde{p}^{ls}_{1,2} =0,\tilde{p}^{ls}_{2,1}=3,\tilde{p}^{ls}_{2,2}=0, \\
    & \tilde{p}^{ex}_{1,1} =0,\tilde{p}^{ex}_{1,2} =1/2,\tilde{p}^{ex}_{2,1}=0,\tilde{p}^{ex}_{2,2}=0, \\  &\tilde{u}_{1,1}=0, \tilde{u}_{1,2}=3/4, \tilde{u}_{2,1}=0, \tilde{u}_{2,2}=0, 
\end{align*}
with the optimal objective value $\tilde{v}=7/2+2\tau$.  Therefore, $\tilde{v} <  \hat{v}$ for all $\tau\in[0,1/2)$. Hence, \eqref{formulation:regularization} problem and \eqref{formulation:regularization_lp} problem are not equivalent.
    \QEDB
\end{example}

In this section, we have shown that \eqref{formulation:regularization} problem is easy to solve, since there is no integrality gap between this problem and its linear programming relaxation. In the next two subsections, we begin to analyze the relationship between \eqref{formulation:regularization} problem and \eqref{formulation:original_mip} problem. 
We would like to understand the differences between the optimal solutions and corresponding objective function values of these two problems both qualitatively and quantitatively. 

\subsection{Structural Properties of the Regularized MIP Model}
In this section, we provide some structural properties of \eqref{formulation:regularization} problem. First, we baseline the value of $\bm \lambda$. Next, we present a two-part result on the structure of an optimal solution \rev{to} the \eqref{formulation:regularization} problem as a function of the penalty coefficients $\bm \lambda$ that distinguishes \eqref{formulation:original_mip} problem and \eqref{formulation:regularization} problem.
\subsubsection{Baselining the value of \texorpdfstring{$\bm{\lambda}$}{Lg}.}
We begin with the following standard observation from linear programming applied to the convex hull of the feasible region of the \eqref{formulation:regularization}.
\begin{remark}\label{lemma_f_reg_increasing}
Function $\rev{z^{\mathrm{reg}}(\bm \lambda)}$ is concave and monotone nondecreasing with respect to $\bm \lambda\in \Re_+^2$. 
\end{remark}

As $\bm \lambda$ gets larger, \eqref{formulation:regularization} problem gets more-and-more ``different" from 
\eqref{formulation:original_mip} problem. In particular, we expect that the battery to be used less, since it now costs more to charge or discharge. However, what is a ``reasonable" value of $\bm \lambda$?  Our first result below allows us to baseline the value of $\bm \lambda$, by showing that if both components of $\bm \lambda$ are equal to $1$ or higher, then \eqref{formulation:regularization} problem effectively solves the problem with no batteries placed in the network.

\begin{restatable}{proposition}{propositionlambdabiggerthan1}\label{proposition:lambda_bigger_than_1}
For any $\bm \lambda \ge \e$, we have $\rev{z^{\mathrm{reg}}(\bm \lambda)} = \rev{z^{\mathrm{nb}}}$ where 
\begin{align}
    \rev{z^{\mathrm{nb}}} =  \min_{\bm \theta, \bm f,\bm p,\bm u} \left\{ c(\bm p) \colon \eqref{eq:generation_limits}- \eqref{eq:power_balance},\; \bm u \in \{0,1\}^{T \times N},\;\bm p^c = \bm p^d =\bm 0 \right\} .\label{formulation:no_battery}
\end{align}
\end{restatable}

\begin{proof}
Observe that the problem corresponding to $\rev{z^{\mathrm{nb}}}$ is obtained by restricting the feasible region of $\rev{z^{\mathrm{reg}}(\bm \lambda)}$ to $\bm p^c = \bm p^d = \bm 0$.
Thus, we have
\begin{align*}
    \rev{z^{\mathrm{nb}}} \ge \rev{z^{\mathrm{reg}}(\bm \lambda)}, \quad \forall \bm \lambda \geq \e.
\end{align*}
To show the opposite inequality, it is sufficient to show that there is an optimal solution for $\rev{z^{\mathrm{reg}}(\e)}$ such that $\bm p^c = \bm p^d = \bm 0$, since we have  $\rev{z^{\mathrm{reg}}(\bm \lambda)} \ge \rev{z^{\mathrm{reg}}(\e)}$ for all $\bm \lambda \ge \e$ from \Cref{lemma_f_reg_increasing}.
We use the power balance equation  \eqref{eq:power_balance} written in the following form
\begin{align*}
    \sum_{j \in \delta^{+}_{i}} f_{t,ij} - \sum_{j \in \delta^-_{i}} f_{t,ji} - p^{g}_{t,i} + D_{t,i} = - p^{c}_{t,i} + p^{d}_{t,i} + p^{ls}_{t,i} - p^{ex}_{t,i}, && \forall  t \in \T, i \in \N.
\end{align*}
To construct such a feasible solution from the current optimal solution. Let ($\hat{\bm \theta}, \hat{\bm f},\hat{\bm p},\hat{\bm u}$) be an optimal solution \rev{to} \eqref{formulation:regularization} problem. We can construct another feasible solution ($\tilde{\bm \theta},\tilde{\bm f},\tilde{\bm p},\tilde{\bm u}$) as follows:
\begin{align*}
    & \tilde{\bm \theta} = \hat{\bm \theta}, \; \tilde{\bm f} = \hat{\bm f}, \; \tilde{\bm p}^g = \hat{\bm p}^g, \; \tilde{\bm p}^s = E^0\e, \; \tilde{\bm p}^c = \bm 0, \; \tilde{\bm p}^d = \bm 0,\\
    & \tilde{p}^{ls}_{t,i} = \max \{-\hat{p}^c_{t,i} + \hat{p}^d_{t,i} + \hat{p}^{ls}_{t,i} - \hat{p}^{ex}_{t,i}, 0\}, && \forall  t \in \T, i \in \N,  \\
    & \tilde{p}^{ex}_{t,i} = \max \{\hat{p}^c_{t,i} - \hat{p}^d_{t,i} - \hat{p}^{ls}_{t,i} + \hat{p}^{ex}_{t,i}, 0\}, && \forall t \in \T, \; i \in \N.
\end{align*}
Then, using the fact that $p^{ls}_{t,i}  p^{ex}_{t,i} = 0$ and $p^{c}_{t,i}  p^{d}_{t,i}=0$ for each $ t \in \T, i \in \N$, we consider the following four cases.
\begin{enumerate}[label={(Case \arabic*)},leftmargin=*]
    \item When $\hat{p}^{d}_{t,i}=\hat{p}^{ls}_{t,i}=0$, $\tilde{p}^{ls}_{t,i}=0$ and $\tilde{p}^{ex}_{t,i}=\hat{p}^c_{t,i} + \hat{p}^{ex}_{t,i}$. Hence, we have
    \begin{align*}
        & \tilde{p}^{ls}_{t,i} + \tilde{p}^{ex}_{t,i}=\hat{p}^c_{t,i} + \hat{p}^{ex}_{t,i}=\hat{p}^{ls}_{t,i} + \hat{p}^{ex}_{t,i} + \hat{p}^c_{t,i} + \hat{p}^d_{t,i}.
    \end{align*}
    \item When $\hat{p}^{d}_{t,i}=\hat{p}^{ex}_{t,i}=0$, $\tilde{p}^{ls}_{t,i} + \tilde{p}^{ex}_{t,i}=|\hat{p}^{ls}_{t,i}-\hat{p}^{c}_{t,i}|$ and we have
    \begin{align*}
        & \tilde{p}^{ls}_{t,i} + \tilde{p}^{ex}_{t,i}
        = |\hat{p}^{ls}_{t,i}-\hat{p}^{c}_{t,i}|
        \le \hat{p}^{ls}_{t,i} + \hat{p}^{ex}_{t,i} + \hat{p}^c_{t,i} + \hat{p}^d_{t,i}. 
    \end{align*}
    \item When $\hat{p}^{c}_{t,i}=\hat{p}^{ls}_{t,i}=0$, $\tilde{p}^{ls}_{t,i} + \tilde{p}^{ex}_{t,i}=|\hat{p}^{ex}_{t,i}-\hat{p}^{d}_{t,i}|$ and we have
    \begin{align*}
        & \tilde{p}^{ls}_{t,i} + \tilde{p}^{ex}_{t,i}
        = |\hat{p}^{ex}_{t,i}-\hat{p}^{d}_{t,i}|
        \le \hat{p}^{ls}_{t,i} + \hat{p}^{ex}_{t,i} + \hat{p}^c_{t,i} + \hat{p}^d_{t,i}.
    \end{align*}
    \item When $\hat{p}^{c}_{t,i}=\hat{p}^{ex}_{t,i}=0$, $\tilde{p}^{ex}_{t,i}=0$ and $\tilde{p}^{ls}_{t,i}=\hat{p}^d_{t,i} + \hat{p}^{ls}_{t,i}$. Hence, we have
    \begin{align*}
        & \tilde{p}^{ls}_{t,i} + \tilde{p}^{ex}_{t,i}=\hat{p}^d_{t,i} + \hat{p}^{ls}_{t,i}=\hat{p}^{ls}_{t,i} + \hat{p}^{ex}_{t,i} + \hat{p}^c_{t,i} + \hat{p}^d_{t,i}.
    \end{align*}
\end{enumerate}
Since the above four cases hold for all $t \in \T, i \in \N$, we have 
\rev{
\begin{align*}
    c(\tilde{\bm p}) + g(\tilde{\bm p}) 
    & =   c^g(\bm{\tilde p^g}) + \sum_{t \in \mathcal{T}} \sum_{i \in \mathcal{N}} \left(\tilde{p}^{ls}_{t,i} + \tilde{p}^{ex}_{t,i} \right) \\
    & \le c^g(\bm{\hat p^g}) + \sum_{t \in \mathcal{T}} \sum_{i \in \mathcal{N}} \left( \hat{p}^{ls}_{t,i} + \hat{p}^{ex}_{t,i} + \hat{p}^{c}_{t,i} + \hat{p}^{d}_{t,i} \right) = c(\hat{\bm p}) + g(\hat{\bm p}).
\end{align*}
}
Therefore, ($\tilde{\bm \theta},\tilde{\bm f},\tilde{\bm p},\tilde{\bm u}$) is an optimal solution \rev{to} \eqref{formulation:regularization} problem. This completes the proof.
\QEDA
\end{proof}

We remark that for general values of $\bm \lambda$, the result of \Cref{proposition:lambda_bigger_than_1}, that is \rev{$z^{\textrm{nb}} =z^{\textrm{reg}}(\bm \lambda)$} may not hold. The following example shows that the battery may always be used when $\bm\lambda \in (0,1)^2$. 

\begin{example}\rm
\label{ex:pcpd_always_positive}
\textit{($\rev{z^{\mathrm{reg}}(\bm \lambda) \neq z^{\mathrm{nb}}}$ when $\bm \lambda \in (0,1)^2$)} We consider the same network as that in Example~\ref{ex:E_min_not_zero} but with different battery configurations and demands.
One battery is placed at node $2$ (i.e., $\N_b=\{2\}$) with $E_c^{\min} = E_d^{\min}=0, E_c^{\max} = E_d^{\max}=2$, $E^{\min}=0$, $E^{\max}=20$, $E_0=20$ and \rev{$\eta_c = \eta_d =1$}. 
Demand is $D_{1,1}=D_{1,2}=D_{2,1}=D_{2,2}=5$.
Other parameters remain the same.
When $\bm \lambda\in(0,1)^2$, at optimality of \eqref{formulation:regularization} problem, we always have
\begin{align*}
    p^d_{1,1}=   p^d_{2,1} =2 >0.
\end{align*}
This demonstrates that when violating the condition in \Cref{proposition:lambda_bigger_than_1}, i.e., $\bm \lambda \in (0,1)^2$, it is possible that either $\bm p^c$ or $\bm p^d$ is always positive.
    \QEDB
\end{example}

\subsubsection{Structural properties of optimal solutions of Regularized MIP.}
The next two-part result on the structure of an optimal solution \rev{to} \eqref{formulation:regularization} problem as a function of the penalty coefficients $\bm \lambda$ shows the distinction between  \eqref{formulation:original_mip} problem and \eqref{formulation:regularization} problem as a function of $\bm \lambda$.

In the first part, we show that $p^c_{t,i} p^{ls}_{t,i} = 0$ for all $ t \in \T,\;i \in \N$ holds for optimal solutions for all values of $\bm \lambda$, that is, this property is true for both \eqref{formulation:original_mip} problem and \eqref{formulation:regularization} problem. Intuitively, this property holds because when the system is incurring load shedding ($p^{ls}_{t,i} > 0$), it would not create additional load shedding by charging a battery ($p^{c}_{t,i} > 0$). Similar to the result above, we may expect that when there is excess power ($p^{ex}_{t,i} > 0$), the amount of discharge would not be positive ($p^{d}_{t,i} = 0$), as a positive discharge amount would further increase excess power. It is reasonable to expect $p^d_{t,i} p^{ex}_{t,i} = 0$ for all $ t \in \T,\;i \in \N$. In the second part of our result, we demonstrate that this condition only holds when $\bm \lambda$ is sufficiently large. Indeed, it turns out that when $\bm \lambda = \bm 0$, specifically in considering the \eqref{formulation:original_mip} problem, the aforementioned condition might be violated, i.e., $p^d_{t,i} p^{ex}_{t,i} > 0$ for some $ t \in \T,\;i \in \N$.

 \begin{restatable}{theorem}{thmstructure}\label{thm:structure}
Suppose $E^{\min}_c = E^{\min}_d = 0$. Let $\bm p$ be an optimal solution to \eqref{formulation:regularization} problem.  Then:
    \begin{itemize}
    \item[(i)] For all $\bm \lambda \in \mathbb{R}^2_+$, we have $p^c_{t,i} p^{ls}_{t,i} = 0$ for all $ t \in \T,\;i \in \N$.
    \item[(ii)] If \rev{$\lambda_c + \eta_c \eta_d \lambda_d > 1-\eta_c \eta_d$}, then we have $p^d_{t,i} p^{ex}_{t,i} = 0$ for all $t \in \T,\;i \in \N$. 
    \end{itemize}
\end{restatable}
\begin{proof}
    See Appendix~\ref{proof_thmstructure}.
    \QEDA
\end{proof}
Notice that the observations in \Cref{thm:structure} can be found in standard IEEE networks with reasonable efficiency levels (\rev{$\eta_c, \eta_d \ge 0.8$}), wherein batteries placed at specific nodes may discharge while the system may incur excess power simultaneously. 
Below, we also present an example that illustrates this phenomenon for a simpler network, thus establishing structural differences in the optimal solutions of \eqref{formulation:original_mip} problem and \eqref{formulation:regularization} problem for sufficiently large $\bm \lambda$.

\begin{example}\rm
\label{ex:pdpex_not_zero}
\textit{(Condition (ii) in \Cref{thm:structure})} Consider a simple network with $\N=\{1,2\}$, $\T=\{1,2,3\}$, $\L=\{(1,2)\}$. 
Suppose one battery is placed at node $2$ (there is no battery placed at node $1$) with $\N_b=\{2\}$ and $E_c^{\min} = E_d^{\min}=0$, $E_c^{\max} = E_d^{\max}=2$, $E^{\min}=0$, $E^{\max}=4$, $E_0=4$, and \rev{$\eta_c = \eta_d=0.1$}. 
Assume each node has one generator with $G_1^{\min}=G_2^{\min} =2$ and $G_1^{\max}=G_2^{\max} = 4$ \rev{and generator costs $c^g({\bm p^g})=0$}. 
We further assume $-4\leq f_{12} \leq 4$ and the demand is $D_{1,1}=2,\;D_{1,2}={2},\;D_{2,1}={1},\;D_{2,2}={1},\;D_{3,1}={2},\;D_{3,2}={1} $. 
Without loss of generality, we assume that the Ohm's law constraint \eqref{eq:dc_power_flow} is satisfied. An optimal solution \rev{to} \eqref{formulation:original_mip} problem (i.e., $\bm \lambda=\bm 0$ in \eqref{formulation:regularization} problem), denoted as $(\bm p^*,\bm u^*)$, is 
\begin{align*}
   & {p^c}^*_{1,1} =0,\;{p^c}^*_{1,2} =0,\;{p^c}^*_{2,1}=0,\;{p^c}^*_{2,2}=2,\;{p^c}^*_{3,1}=0,\;{p^c}^*_{3,2}={1}, \\
    & {p^d}^*_{1,1} =0,\;{p^d}^*_{1,2} ={0.03},\;{p^d}^*_{2,1}=0,\;{p^d}^*_{2,2}=0,\;{p^d}^*_{2,1}=0,\;{p^d}^*_{2,2}=0, \\
       & {p^{ls}}^*_{1,1} =0,\;{p^{ls}}^*_{1,2} =0,\;{p^{ls}}^*_{2,1}=0,\;{p^{ls}}^*_{2,2}=0,\;{p^{ls}}^*_{3,1}=0,\;{p^{ls}}^*_{3,2}=0, \\
            & {p^{ex}}^*_{1,1} =0,\;{p^{ex}}^*_{1,2}={0.03},\;{p^{ex}}^*_{2,1}=0,\;{p^{ex}}^*_{2,2}=0, \;{p^{ex}}^*_{3,1}=0,\;{p^{ex}}^*_{3,2}=0,\\  
    & u^*_{1,1} =0,\;u^*_{1,2}=0,\;u^*_{2,1}=0,\;u^*_{2,2}=1,\;u^*_{3,1}=0,\;u^*_{3,2}=1.   \end{align*}
Clearly, in this example, ${p^d}^*_{1,2} {p^{ex}}^*_{1,2}>0$. However, we can avoid this situation after considering regularization. When $\bm \lambda=(0.99,\;0.99)^\top$, i.e., this particular choice of $\bm \lambda$ satisfies Condition (ii) in \Cref{thm:structure}, an optimal solution \rev{to} \eqref{formulation:regularization} problem, denoted as $(\hat{\bm p},\hat{\bm u})$, is 
\begin{align*}
   & \hat{p}^c_{1,1} =0,\;\hat{p}^c_{1,2} =0,\;\hat{p}^c_{2,1}=0,\;\hat{p}^c_{2,2}=0,\;\hat{p}^c_{3,1}=0,\;\hat{p}^c_{3,2}={0}, \\
    & \hat{p}^d_{1,1} =0,\;\hat{p}^d_{1,2} =0,\;\hat{p}^d_{2,1}=0,\;\hat{p}^d_{2,2}=0,\;\hat{p}^d_{2,1}=0,\;\hat{p}^d_{2,2}=0, \\
       & \hat{p}^{ls}_{1,1} =0,\;\hat{p}^{ls}_{1,2} =0,\;\hat{p}^{ls}_{2,1}=0,\;\hat{p}^{ls}_{2,2}=0,\;\hat{p}^{ls}_{3,1}=0,\;\hat{p}^{ls}_{3,2}=0, \\
            & \hat{p}^{ex}_{1,1} =0,\;\hat{p}^{ex}_{1,2}=0,\;\hat{p}^{ex}_{2,1}=2,\;\hat{p}^{ex}_{2,2}=0, \;\hat{p}^{ex}_{3,1}=0,\;\hat{p}^{ex}_{3,2}={1},\\  
    & \hat{u}_{1,1} =0,\;\hat{u}_{1,2}=0,\;\hat{u}_{2,1}=0,\;\hat{u}_{2,2}=0,\;\hat{u}_{3,1}=0,\;\hat{u}_{3,2}={0}.   \end{align*}
In this example, Condition (i) in \Cref{thm:structure} is satisfied for both \eqref{formulation:original_mip} problem and \eqref{formulation:regularization} problem whereas Condition (ii) is only satisfied for \eqref{formulation:regularization} problem. 
    \QEDB
\end{example}

To conclude the discussions in this subsection, we provide a summary of the choice of $\bm\lambda$ in \Cref{figure_lambda_comparison}.
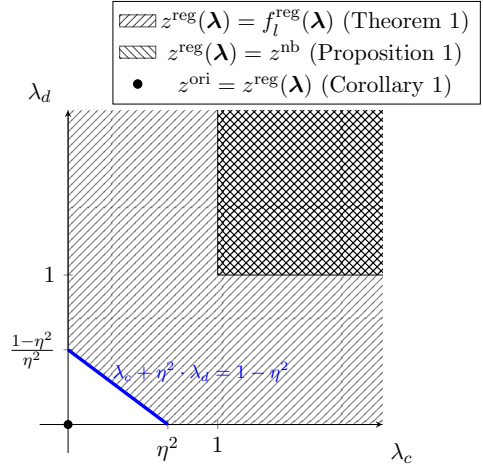
\begin{figure}[h]\centering
\begin{tikzpicture}[scale=0.8]
    \begin{axis}[scale=1,
        unit vector ratio*=1 1 1,
        domain=-0.19:2.10,
        xmin=-0.19, xmax=2.10,
        ymin=-0.19, ymax=2.10,
        samples=200,
        axis y line=center,
        axis x line=middle,
        xtick={0,2/3,1},
        ytick={0,1/2,1},
        xticklabels={0, $\eta^2$, 1}, 
        yticklabels={0, $\frac{1-\eta^2}{\eta^2}$,1},
        xlabel={$\lambda_c$},
        ylabel={$\lambda_d$},
        x label style={
            at={(1,0)}, 
            anchor=west, 
        },
        y label style={
            at={(0,1)}, 
            anchor=south, 
        },
        legend style={at={(0.75,1.32)},anchor=north},
        ]
        \filldraw [pattern=north east lines,opacity=0.7](0,5) -- (0,1/2) -- (2/3,0)-- (5,0);
        \addlegendimage{area legend,pattern=north east lines, opacity=0.7}
        \addlegendentry{$\rev{z^{\mathrm{reg}}(\bm \lambda)} =  f^{\mathrm{reg}}_{l}(\bm \lambda)$ (\Cref{theorem:ip=lp_condition})}
            
        \filldraw [pattern=crosshatch](1,3) -- (1,1) -- (3,1)-- (5,5);
        \addlegendimage{area legend,pattern=north west lines, opacity=0.7}
        \addlegendentry{$\rev{z^{\mathrm{reg}}(\bm \lambda)} = \rev{z^{\mathrm{nb}}}$ (\Cref{proposition:lambda_bigger_than_1})}

        \node [circle,fill,inner sep=1.5pt] at (0,0) {};
        \addlegendimage{only marks, mark=*}
        \addlegendentry{$\rev{z^{\mathrm{ori}}} =  \rev{z^{\mathrm{reg}}(\bm \lambda)}$ (\Cref{corollary:fori=freg})}

        \addplot+[mark=none,blue,-][domain=0:2/3] {1/2-3/4*x} node[pos=0.3](blue1){};
        \node [right,color=blue,-] at (blue1) {\footnotesize\footnotesize\footnotesize{ $\lambda_c + \eta^2\cdot \lambda_d = 1-\eta^2$}}; 
        \draw [ultra thick, blue, - ] (2/3,0) -- (0,1/2);
    \end{axis}	
\end{tikzpicture}
\caption{Model comparisons with different choices of $\lambda_c$ and $\lambda_d$.} 
\label{figure_lambda_comparison}
\end{figure}
    
\subsection{Error Quantification of the Solution from the Regularized MIP Model}
\label{sec:error_quantification}
As we have seen in \Cref{theorem:ip=lp_condition}, in \eqref{formulation:regularization} problem, the penalty required to have zero integrality gap with the LP relaxation decreases when efficiency $\eta$ gets closer to $1$. Nevertheless, we still aim for a better understanding of the quality of objective function change when we adjust the regularizer $\bm \lambda$. Hence, the goal of this subsection is to discuss analytical differences in comparison to the original model with respect to the solution quality.

First, we present a sufficient condition under which the optimal solution \rev{to} \eqref{formulation:regularization} can be used to recover the optimal battery operation schedule for the original problem. 
\begin{restatable}{proposition}{propositionexactnesscondition}\label{proposition:exactness_condition}
Let $\U^*= \{ \bm u : \exists \; \bm \theta, \; \bm f, \;\bm p$ such that together $(\bm \theta, \bm f, \bm p, \bm u)$ is an optimal solution \rev{to} \eqref{formulation:original_mip} problem$\}$ and $\P^*= \{ \bm p : \exists \; \bm \theta, \; \bm f, \;\bm u$ such that together $(\bm \theta, \bm f, \bm p, \bm u)$ is an optimal solution \rev{to} \eqref{formulation:original_mip} problem$\}$. Define the second-best optimal objective value of \eqref{formulation:original_mip} problem as $\rev{z^{\mathrm{ori}}}(\U^*) = \min_{\bm \theta, \bm f,\bm p, \bm u} \left\{ c(\bm p) \colon \eqref{eq:generation_limits}-\eqref{eq:power_balance},\: \bm u \in\{0,1\}^{T \times N}\setminus \U^*\right\}.$ Define the difference between the best optimal objective value and the second best optimal objective value as $\delta = |\rev{z^{\mathrm{ori}}}(\U^*) - \rev{z^{\mathrm{ori}}}| > 0$. Suppose $(\hat{\bm \theta}, \hat{\bm f}, \hat{\bm p}, \hat{\bm u})$ is an optimal solution to \eqref{formulation:regularization} problem. If $\bm \lambda^\top g(\bm p^*) < \delta $ for some $\bm p^* \in \P^*$, then $\hat{\bm u} \in \U^*$.
\end{restatable}

\begin{proof}
We prove this by contradiction. Suppose that the presumptions hold and $\hat{\bm u} \notin \U^*$. Then, $(\hat{\bm \theta}, \hat{\bm f}, \hat{\bm p}, \hat{\bm u})$ is a feasible but not an optimal solution to \eqref{formulation:original_mip} problem. Since $\hat{\bm u} \in \{0,1\}^{T \times N} \setminus \U^*$, we have $c(\hat{\bm p}) \ge \rev{z^{\mathrm{ori}}}(\U^*)$. Let $({\bm \theta^*}, {\bm f^*}, {\bm p^*}, {\bm u^*})$ be an optimal solution to \eqref{formulation:original_mip} problem such that $
\bm\lambda^\top g({\bm p^*}) = \min_{\bm p \in \P^*} \bm\lambda^\top g(\bm p)$. Then by the definition of $\delta$, we have:
\begin{align*}
    c(\hat{\bm p}) - c({\bm p^*}) \ge |\rev{z^{\mathrm{ori}}}(\U^*) - \rev{z^{\mathrm{ori}}}| = \delta.
\end{align*}
According to the optimality condition from \eqref{formulation:regularization} problem, we have
\begin{align*}
    c(\hat{\bm p}) + \bm \lambda^\top g(\hat{\bm p})  \le 	c({\bm p^*})+ \bm \lambda^\top g({\bm p^*}).  
\end{align*}
Rearranging the terms, we have
\begin{align*}
    c(\hat{\bm p}) - c({{\bm p^*}}) \le \bm \lambda^\top g({\bm p^*}) - \bm \lambda^\top g(\hat{\bm p}) \le \bm \lambda^\top g({\bm p^*}) < \delta.
\end{align*}
Clearly, $c(\hat{\bm p}) - c({{\bm p^*}}) \ge \delta$ and $c(\hat{\bm p}) - c({{\bm p^*}}) < \delta$ cannot hold simultaneously. Hence, this is a contradiction.
\QEDA
\end{proof}
Notice that when $\bm \lambda$ is small (when $\eta$ is close to 1), the sufficient condition of Propostion~\ref{proposition:exactness_condition} is easy to satisfy. In our computational experiments, we often see this behavior. 
We provide an example to illustrate the exactness condition in Proposition~\ref{proposition:exactness_condition}.

\begin{example}\rm
\label{ex:exactness_condition}
\textit{(Exactness Condition of \Cref{proposition:exactness_condition})} Consider a simple network with $\N=\{1,2\}$, $\T=\{1,2\}$, $\L=\{(1,2)\}$. 
Suppose one battery is placed at node $2$ (there is no battery placed at node $1$) with $\N_b=\{2\}$ and $E_c^{\min} = E_d^{\min}=0$, $E_c^{\max} = E_d^{\max}=2$, $E^{\min}=0$, $E^{\max}=4$, $E_0=2$, and \rev{$\eta_c = \eta_d =0.9$}. 
Assume each node has one generator with $G_1^{\min}=G_2^{\min} =2$ and $G_1^{\max}=G_2^{\max} = 4$ \rev{and generator costs $c^g({\bm p^g})=0$}. 
We further assume $-4\leq f_{12} \leq 4$ and the demand is $D_{1,1}=10,\;D_{1,2}=4,\;D_{2,1}=4,\;D_{2,2}=4$. 
Without loss of generality, we assume that the Ohm's law constraint \eqref{eq:dc_power_flow} is satisfied.
The optimal objective value of the \eqref{formulation:original_mip} problem $f^{\textrm{ori}} = 4.2$. We enumerate all optimal solutions $\bm u$ that achieve this value and find that $\U^*=\{\hat{ \bm u},\bar{\bm u}\}$ with 
\begin{align*}
    & \bar{u}_{1,1} =0,\;\bar{u}_{1,2}=0,\;\bar{u}_{2,1}=0,\;\bar{u}_{2,2}=0,\\
    &  \hat{u}_{1,1} =0,\;\hat{u}_{1,2}=0,\;\hat{u}_{2,1}=0,\;\hat{u}_{2,2}=1.
\end{align*}
Excluding the solutions in $\U^*$, the second-best optimal objective value $\rev{z^{\mathrm{ori}}}(\U^*)=6$. Hence, the difference $\delta=1.8$. When $\bm \lambda = \left({1-\eta^2}/{1+\eta^2},{1-\eta^2}/{1+\eta^2}\right)^\top = \left({19}/{181},{19}/{181}\right)^\top$, $\argmin\{ g(\bm p^*):p^* \in \P^*\} = [0,1.8]^\top$. Notice that $\bm \lambda^\top g(\bm p^*) < \delta$. Therefore, the optimal solution \rev{to} \eqref{formulation:regularization} problem should be exactly the \eqref{formulation:original_mip} problem. We check this condition by solving \eqref{formulation:regularization} problem and we confirm that the solution from \eqref{formulation:regularization} problem is indeed exact.
    \QEDB
\end{example}

In general, we may not be able to show that the solution of \eqref{formulation:regularization} problem recovers a solution to \eqref{formulation:original_mip} problem \rev{as verifying the condition $\bm \lambda^\top g(\bm p^*)$ in \Cref{proposition:exactness_condition} is challenging and would almost imply that we can solve the original problem directly}. Instead, we next provide a bound that quantifies the difference between the objective function value $c(\bm p)$ obtained from the optimal solution \rev{to} \eqref{formulation:original_mip} problem and that obtained from the optimal solution \rev{to} \eqref{formulation:regularization} problem. 

\begin{restatable}{theorem}{theoremworstbound}\label{theorem:worst_bound}
Let $(\bm \theta^*,{\bm f^*},{\bm p^*},{\bm u^*})$ be the optimal solution \rev{to} \eqref{formulation:original_mip} problem and $(\hat{\bm \theta},\hat{\bm f},\hat{\bm p},\hat{\bm u})$ the optimal solution \rev{to} \eqref{formulation:regularization} problem with regularizer $(\lambda_c,\; \lambda_d)$.
The gap between \eqref{formulation:original_mip} problem and \eqref{formulation:regularization} problem is
\begin{align}
\label{eq:worst_bound}
  c(\hat{\bm p}) - c(\bm p^*) \leq T N_b \max \{ E_c^{\max} \lambda_c, E_d^{\max}\lambda_d\}. 
\end{align}
\end{restatable}

\begin{proof}
Notice that the feasible regions are the same in \eqref{formulation:original_mip} problem and \eqref{formulation:regularization} problem. 
Hence, $(\hat{\bm \theta},\hat{\bm f},\hat{\bm p},\hat{\bm u})$ is a feasible solution to \eqref{formulation:original_mip} problem and $(\bm \theta^*,{\bm f^*},{\bm p^*},{\bm u^*})$ is a feasible solution to \eqref{formulation:regularization} problem.
By optimality, we have
\begin{align*}
    c(\hat{\bm p}) + \bm \lambda^\top g(\hat{\bm p}) & \le c(\bm p^*) + \bm \lambda^\top g(\bm p^*),
\end{align*}
which implies that
\begin{equation}
\begin{aligned}
\label{eq:prop3_dropping_g}
    c(\hat{\bm p}) - c(\bm p^*) 
    \le \bm \lambda^\top g(\bm p^*) - \bm \lambda^\top  g(\hat{\bm p}) 
    \le \bm \lambda^\top g(\bm p^*). \notag \\
\end{aligned}
\end{equation}
Recall that $\bm \lambda^\top g(\bm p^*) = \sum_{t \in \T} \sum_{i \in \N}\left[ \lambda_c p_{t,i}^{c*} + \lambda_d p_{t,i}^{d*}\right]$. Since $p_{t,i}^{c*}p_{t,i}^{d*}=0$ for all $ t \in \T, i \in \N$, it follows $\lambda_c p_{t,i}^{c*} + \lambda_d p_{t,i}^{d*} \le \max \{E_c^{\max}\lambda_c , E_d^{\max}\lambda_d \}$ for all $t \in \T, i \in \N$.
From \eqref{eq:no_battery}, $p_{t,i}^{c*} = p_{t,i}^{d*}=0$ for all  $ t \in \T,\; i \notin N_b$, we obtain the result in \eqref{eq:worst_bound}.
\QEDA
\end{proof}

Our next goal is then to minimize the worst-case bound predicted by Theorem~\ref{theorem:worst_bound} by selecting specific values for  $\bm \lambda$.

\begin{restatable}{proposition}{propositionbestworstbound}\label{proposition:best_worst_bound}
The best worst-case bound is \rev{achieved} at $\bm \lambda = \left( \frac{E_d^{\max} - \eta^2 E_d^{\max}}{E_d^{\max} + \eta^2 E_c^{\max}}, \; \frac{E_c^{\max} - \eta^2 E_c^{\max}}{E_d^{\max} + \eta^2 E_c^{\max}} \right)^\top$.
\end{restatable}

\begin{proof}
Since $T$ and $N_b$ are fixed, minimizing the worst-case bound in \Cref{theorem:worst_bound} reduces to the following minimization problem:
\begin{align*}
     \min_{\lambda_c\geq 0, \lambda_d\geq 0} \left\{ \max \{E_c^{\max} \lambda_c, E_d^{\max}\lambda_d \}\colon \rev{\lambda_c +  \eta_c \eta_d \lambda_d \ge 1-\eta_c \eta_d} \right\},
\end{align*}
where the optimal objective value is achieved when $E_c^{\max} \lambda_c = E_d^{\max} \lambda_d$. Hence, an optimal solution $(\lambda_c^{*}, \lambda_d^{*})$ is 
\begin{align*}
    (\lambda_c^{*}, \lambda_d^{*}) = 
    \left( 
   \rev{ \frac{E_d^{\max} - \eta_c \eta_d E_d^{\max}}{E_d^{\max} + \eta_c \eta_d E_c^{\max}}},
   \rev{ \frac{E_c^{\max} - \eta_c \eta_d E_c^{\max}}{E_d^{\max} + \eta_c \eta_d E_c^{\max}}}
    \right).
\end{align*}
This concludes the proof.
\QEDA
\end{proof}
We remark that when $E_c^{\max} = E_d^{\max}$, an optimal choice of $(\lambda_c, \lambda_d)$ reduces to
\[(\lambda_c, \lambda_d) = \rev{\left(\frac{1-\eta_c \eta_d}{1+\eta_c \eta_d},\;\frac{1-\eta_c \eta_d}{1+\eta_c \eta_d}\right)},\]
which implies that for this specific case, the best worst-case bound in \eqref{eq:worst_bound} is achieved when $\lambda_c = \lambda_d$.

\Cref{theorem:worst_bound} and \Cref{proposition:best_worst_bound} provide the worst-case analysis of $c(\hat{\bm p}) - c(\bm p^*)$, where $\bm p^*$ is an optimal solution \rev{to} \eqref{formulation:original_mip} problem and $\hat{\bm p}$ is an optimal solution \rev{to} \eqref{formulation:regularization} problem. 
However, we expect the actual difference between the objective function value of the problem to be much smaller due to the following reasons: (i) In the proof of \Cref{theorem:worst_bound}, we drop the charging and discharging values of the regularized MIP solution by taking the minimum level of zero for all times considered. Even though we expect that the amounts of charging and discharging become smaller when $\bm \lambda$ increases, assuming them to be completely not charging or discharging may be an underestimation (Proposition~\ref{proposition:lambda_bigger_than_1}).
Indeed for high-efficiency values, we expect the amount of charging and discharging in the regularized MIP model to be close to the amount of charging and discharging in the original MIP model;
and (ii) Also note that we upper bound amount of charging and discharging by $ E_c^{\max}$ and $E_d^{\max}$, respectively, which in many cases can be a significant overestimation. Clearly, as shown in \Cref{theorem:worst_bound}, the worst-case bound depends on the values of $E_c^{\max}$ and $E_d^{\max} $. 

To summarize the results of this section, we show that when the efficiency level increases, in the following example, the optimal solution from \eqref{formulation:regularization} problem is an optimal solution to \eqref{formulation:original_mip} problem, and the empirical gap is much smaller than the theoretical gap. Our empirical results in \Cref{sec:numerical_experiment} further demonstrate that the true difference is much less than the theoretical worst-case bound in the above result.

\begin{example}\rm
\label{example:bound}
\textit{(Gap between solutions)} Consider a simple network with $\N=\{1,2\}$, $\T=\{1,2\}$, $\L=\{(1,2)\}$. Suppose one battery is placed at node $2$ (i.e., $\N_b=\{2\}$) with $E_c^{\min} = E_d^{\min}=0$, $E_c^{\max} = E_d^{\max}=1$, $E^{\min}=0$, $E^{\max}=6$, and $E_0=0$. 
Assume each node has one generator with $G_1^{\min}=G_2^{\min} =2$ and $G_1^{\max}=G_2^{\max} = 4$ \rev{and generator costs $c^g({\bm p^g})=0$}. 
We further assume $-1\leq f_{12} \leq 1$, the demand is $D_{1,1}=5,\;D_{1,2}={1},\;D_{2,1}={8},\;D_{2,2}={4}$, and the values of the regularizers $\lambda_c$ and  $\lambda_d$ are the same, i.e., $\lambda = \lambda_c = \lambda_d$.
In Figures~\ref{fig:example_4a}-\ref{fig:example_4c}, we numerically illustrate how the values of the regularizer affect the objective function values $c(\bm p)$, where the vertical axis represents the objective function values $c(\bm p)$, and the horizontal axis represents the regularize values $\lambda$.
Three small incremental efficiency levels $\eta \in \{1/\sqrt{2.1}, 1/\sqrt{2}, 1/\sqrt{1.9}\}$ are considered \rev{where $\eta_c = \eta_d = \eta$}. Based on \Cref{proposition:best_worst_bound}, we also plot the objective function values with $\lambda = {1-\eta^2}/{1+\eta^2}$.  From Figure~\ref{fig:example_4a}, we see that when $\eta = 1/\sqrt{2.1} \approx 0.69$, the optimal solution from the regularized MIP, with the choice of $(\lambda_c,\;\lambda_d)$ such that $\lambda_c + \eta^2 \lambda_d \ge 1-\eta^2$, is indeed not an optimal solution to the original MIP.
In Figure~\ref{fig:example_4c}, we observe that, with a higher efficiency level $\eta =1/\sqrt{1.9}\approx 0.72$, the optimal solution from the regularized MIP is the optimal solution to the original MIP, while the choice of $(\lambda_c,\;\lambda_d)$ satisfying the condition that the LP relaxation of the regularized MIP model is without the integrality gap. 
In Figure~\ref{fig:example_4d}, we show how the actual difference is much less than the theoretical worst-case bound for this example.
\QEDB
\end{example}
\begin{figure}[hbt]
    \centering
    \subfigure[$ \eta=\frac{1}{\sqrt{2.1}}$]{
\includegraphics[width=0.4\textwidth]{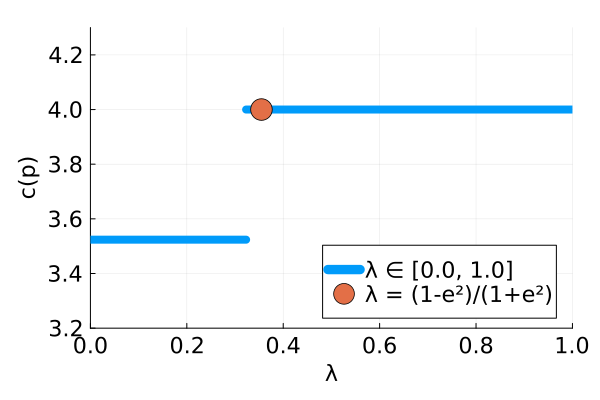}
        \label{fig:example_4a}
    }
    \subfigure[$ \eta=\frac{1}{\sqrt{2}}$]{       \includegraphics[width=0.4\textwidth]{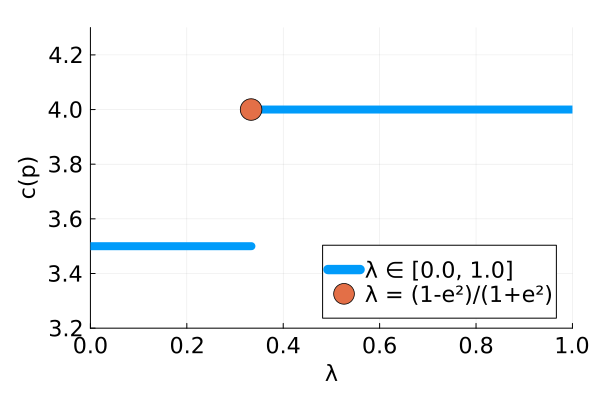}
          \label{fig:example_4b}
    }
    \subfigure[$ \eta=\frac{1}{\sqrt{1.9}}$]{
\includegraphics[width=0.4\textwidth]{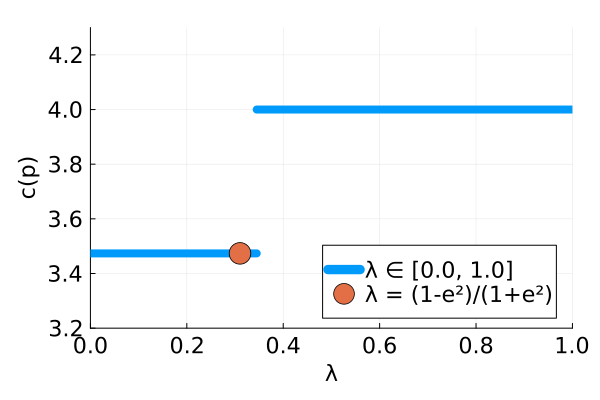}
          \label{fig:example_4c}
    }
    \subfigure[Gap Comparisons]{
\includegraphics[width=0.4\textwidth]{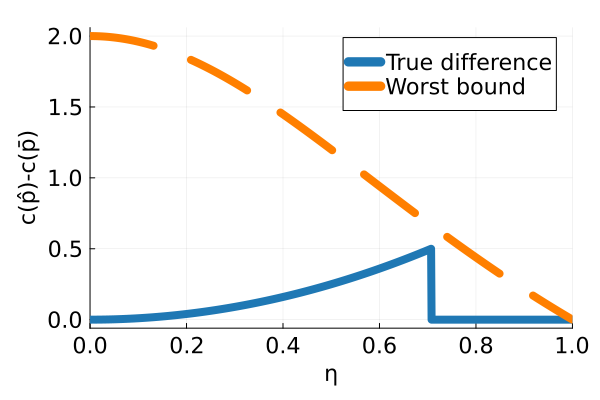}
          \label{fig:example_4d}
    }
    \caption{\rev{Change in the objective value with respect to $\lambda$, where $\lambda_c=\lambda_d=\lambda$ in Proposition~\ref{proposition:best_worst_bound}. The point denotes the value with $\lambda=(1-\eta^2)/(1+\eta^2)$ for Figures (a), (b), and (c). For Figure (d), the orange dashed line represents the theoretical worst-case bound where the blue solid line shows the empirical difference.}}
    \label{fig:example_4}
\end{figure}

\section{Trilevel \texorpdfstring{$N-k$}{Lg} Contingency Problem with ESS Siting}
\label{sec:trilevel_n-k_contingency}
We apply the regularized MIP model to the long-term planning with $N-k$ contingency problem to improve transmission reliability by strategically siting batteries. Battery placement problem under $N-k$ contingency is modeled as a $\min$-$\max$-$\min$ problem with binary variables in all three levels. To the best of our knowledge, there is no known efficient algorithm to solve such a trilevel problem with binary variables at all three levels. We show that \eqref{formulation:regularization} problem provides provably high-quality solutions. 

\subsection{Formulation}
The trilevel $\min$-$\max$-$\min$ problem can be also understood as a defender-attacker-defender problem. A network designer makes a long-term decision on whether to install an energy storage system for node $i\in \N$, represented by variable $x_i$, to enhance the robustness of the power system. The system operator has a budget of $b$ batteries to add. 
The second level is an interdictor who can disrupt up to $k$ transmission lines with the goal of maximizing load shedding or excess power. The lowest level is a system operator solving DCOPF. The overall trilevel problem is presented in Figure \ref{fig:trilevel_summary}. 
For simplicity, following the theoretical and computational results in~\cite{johnson2022scalable}, we remove the Ohm's law constraint from the DCOPF in the third level. \rev{Furthermore, we assume that the generator cost $c^g(\bm p^g) = 0$, since the primary focus of this long-term planning problem under contingency is to meet the demand exactly.}
The detailed formulation is provided in Appendix~\ref{sec:numerical_formulation}.

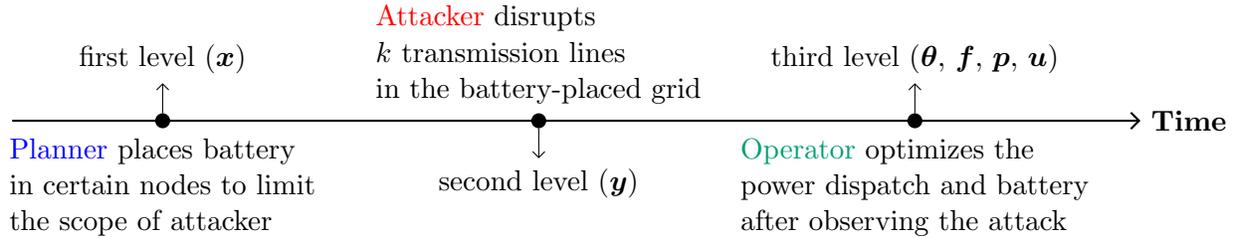
\begin{figure}[!hbt]
    \centering
    \begin{tikzpicture}[
                > = Straight Barb,
    node distance = 28mm and 44mm,
       dot/.style = {circle, fill, inner sep=2pt, outer sep=0pt},
    every label/.append style = {align=left},
    font= \linespread{1}\selectfont
                        ]
    \draw[thick,->] (0,0) -- (15,0) node[right] {\textbf{Time}};
    \node (a) [dot,label=below:\textcolor{blue}{Planner} places battery \\ in certain nodes to limit \\the scope of attacker] at (2.0,0) {};
    \node (b) [dot,label=above:\textcolor{red}{Attacker} disrupts \\ $k$ transmission lines \\in the battery-placed grid] at (7.0,0) {};
    \node (c) [dot,label=below:\textcolor{bluegreen}{Operator} optimizes the \\ power dispatch and battery \\after observing the attack] at (12.0,0) {};
    \draw[->]  (a) -- ++ (0,0.5) node[above] {first level ($\bm x)$};           
    \draw[->]  (b) -- ++ (0,-0.5) node[below] {second level ($\bm y$)};
    \draw[->]  (c) -- ++ (0,0.5) node[above] {third level ($\bm \theta$, $\bm f$, $\bm p$, $\bm u$)};
    \end{tikzpicture}
\caption{Long-term planning under $N-k$ contingency problem is a trilevel $\min$-$\max$-$\min$ problem with binary decision variables in each level.}
\label{fig:trilevel_summary}
\vspace*{-0.5cm}
\end{figure}

\subsection{Solution Methodology}
\label{sec:trilevel_solution_methodology}
There are a few algorithms proposed to solve bilevel or trilevel problems (see, e.g., \citealt{johnson2023covering,zeng2014solving,bienstock2008computing}). However, for this particular trilevel problem with binary variables in all three stages, applying existing methods is computationally intractable. Even for the smallest instance of the IEEE 14-bus system and with the first-stage decision $\bm x$ fixed, solving the bilevel $\max\text{-}\min$ problem using the state-of-the-art algorithm proposed by \cite{zeng2014solving} does not converge in 6 hours. Hence, \rev{applying the existing bilevel algorithms to} the original trilevel problem is intractable using the \eqref{formulation:original_mip} formulation.
However, using the \eqref{formulation:regularization} in the third level has the advantage that the third level becomes an LP per Theorem~\ref{theorem:ip=lp_condition}.
Moreover, we make the following remark.

\begin{remark}
Explicit bounds can be found for all dual variables of the third-level problem that appear in bilinear terms together with the second-level variables.
\end{remark}

The above remark allows the exact linearization for bilinear terms that appear in the objective function using McCormick inequalities. 
We then apply the algorithm by \citet{bienstock2008computing} to solve the resulting bilevel optimization problem.
The details of the boundedness result (Proposition \ref{prop_dual_bound_beta}) and the algorithm are presented in Appendix~\ref{sec:numerical_formulation}. 

Using the regularized MIP model in the third level has another crucial advantage in that it can provide a feasible solution with an upper bound, along with the optimality-gap derived using the LP relaxation of the third level to obtain lower bounds. In particular, the general form of our trilevel problem is the following:
\begin{align}
    z^{OPT} = \min_{\bm x \in \X} \left\{  \max_{\bm y \in \Y} \left\{  \min_{\bm \theta, \bm f, \bm p, \bm u \in \F(\bm x, \bm y)} c(\bm p) \right\} \right\}. \label{eq:trilevel_general}
\end{align}
The set $\Y$ is finite as there are a finite number of edges and only finite possible attack strategies are available. Let $\Y=\{ \bm y^1,\cdots, \bm y^K\}$. Then \eqref{eq:trilevel_general} can be equivalently reformulated as the following:
\begin{equation}
\begin{aligned}
\label{eq:trilevel_OPT}
    z^{OPT} = \min_{\bm x \in \X, \xi} \quad & \xi, \\
    \text{s.t.} \quad & \xi \ge\min \{c(\bm p) : \bm \theta, \bm f, \bm p, \bm u \in \F(\bm x, \bm y^i)\} && \forall i \in [K].
\end{aligned}
\end{equation}
Then, due to Theorem~\ref{theorem:ip=lp_condition}, using the regularized MIP model in the third level is equivalent to solving the following:
\begin{equation}
\begin{aligned}
\label{eq:trilevel_reg}
    z^{REG} = \min_{\bm x \in \X, \xi} \quad & \xi, \\
    \text{s.t.} \quad & \xi \ge \min \{c(\bm p) + \bm \lambda^\top g(\bm p) : \bm \theta, \bm f, \bm p, \bm u \in \F(\bm x, \bm y^i)\} && \forall i \in [K].
\end{aligned}
\end{equation}
Similarly, using the LP relaxation \rev{of the original MIP formulation} in the third level can be formulated as:
\begin{equation}
\begin{aligned}
\label{eq:trilevel_lp}
    z^{LP} = \min_{\bm x \in \X, \xi} \quad & \xi, \\
    \text{s.t.} \quad & \xi \ge \min \{c(\bm p) : \bm \theta, \bm f, \bm p, \bm u \in \R(\bm x, \bm y^i)\} && \forall i \in [K],
\end{aligned}
\end{equation}
where $\R(\bm x, \bm y)$ is a linear relaxation of $\F(\bm x, \bm y)$. It is then easy to verify that $z^{LP} \le z^{OPT} \le z^{REG}$. For example, let $\eta(\bm x,\bm y) = \min \{c(\bm p) : \bm \theta, \bm f, \bm p, \bm u \in \F(\bm x, \bm y)\}$ and $\gamma(\bm x, \bm y) = \min \{c(\bm p) + \bm \lambda^\top g(\bm p) : \bm \theta, \bm f, \bm p, \bm u \in \F(\bm x, \bm y)\}$. 
Let $(\xi^*, \bm x^*, \bm y^*)$ and $(\hat{\xi}, \hat{\bm x}, \hat{\bm y})$ be optimal solutions corresponding to $z^{OPT}$ and $z^{REG}$ respectively. Let $\check{\bm y} \in\textup{arg max}_{\bm y \in \Y}\eta(\hat{ \bm x}, \bm y)$.
Then we have 
{$z^{REG} = \hat{\bm \xi} = \gamma(\hat{\bm x}, \hat{\bm y}) \geq \gamma(\hat{\bm x}, \bm \check{\bm y}) \geq \eta(\hat{\bm x}, \check{\bm y}) \geq \eta (\bm x^*, \bm y^*)  = \bm \xi^* = z^{OPT}$, where the first inequality follows from \eqref{eq:trilevel_reg} and the optimality of $\hat{y}$ for the second-level max optimization problem when $\bm x$ is fixed to $\hat{\bm x}$, the second inequality follows from the definition of $\eta(\cdot)$ and $\gamma(\cdot)$, and the last inequality follows from the optimality of $\bm x^*$ for \eqref{eq:trilevel_OPT}.}
A similar proof can be used to verify that $z^{LP} \leq z^{OPT}$.
Also note that if $\hat{\bm x}$ is the optimal solution of $z^{REG}$, then using the notation from the previous paragraph, we have that  $\max_{\bm y \in \Y} \{  \min_{\bm \theta, \bm f, \bm p, \bm u \in \F(\hat{\bm{x}}, \bm y)} c(\bm p) \} = \eta(\hat{\bm x}, \check {\bm y}) \leq \gamma(\hat{\bm x}, \hat{\bm y}) = z^{REG}.$ Thus, we arrive at the following conclusion.
\begin{proposition} 
\label{proposition:trilevel_max_optimality_gap}
Let $\hat{\bm{x}}$ be an optimal solution of (\ref{eq:trilevel_reg}). Then this solution is a solution for the trilevel problem (\ref{eq:trilevel_OPT}) with an optimality-gap of at most $(z^{REG} - z^{LP}) / z^{REG}$.
\end{proposition}
\rev{Lastly, we consider one of the tightened LP relaxations from \cite{pozo2023convex}, which is referred to as vertex-representation convex hull (VCH):
\begin{equation}
\begin{aligned}
\label{eq:trilevel_vch}
    z^{VCH} = \min_{\bm x \in \X, \xi} \quad & \xi, \\
    \text{s.t.} \quad & \xi \ge \min \{c(\bm p) : \bm \theta, \bm f, \bm p, \bm u \in \H(\bm x, \bm y^i)\} && \forall i \in [K],
\end{aligned}
\end{equation}
where $\H(\bm x, \bm y)$ is the tightened linear relaxation of $\F(x,y)$ using the vertex-representation. Following the same rationale above, we have $z^{VCH} \leq z^{OPT}$.}
\section{Numerical Experiment}
\label{sec:numerical_experiment}
In this section, we demonstrate the strength of the regularized MIP model by (i) testing the quality of the solution from the regularized MIP model against the original MIP model; and (ii) applying the regularized MIP model to solve the long-term planning $N-k$ contingency problem to improve transmission reliability by strategically sitting batteries. The rest of this section is organized as follows. In Section \ref{sec:experimental_setting}, we provide the experimental setting including the description of the networks being used, load scenarios, and relevant battery parameters. 
Sections \ref{sec:testing_regularized_mip} and \ref{sec:computation_results_n-k} present results on the performance of the regularized MIP model to solve DCOPF and trilevel $N-k$ contingency problem respectively. 

\subsection{Experimental Setting}
\label{sec:experimental_setting}
We use standard IEEE instances available from MATPOWER (\citealt{zimmerman2010matpower}) and use the PowerModels package to read the network data (\citealt{powermodels}). All numerical instances are implemented on Julia version 1.7 (\citealt{julia2017}) using Gurobi version 10.0 as the optimization solver (\citealt{gurobi}) on a Linux x86 machine with a 64-bit operating system with 2.3GHz processor on 64GB RAM.  
\subsubsection{Networks.}
The network instance contains generator information, demand load information, and branch information, which is sufficient to model the DCOPF problem with batteries except for the multi-period demand load profiles and battery parameters. In PEGASE and RTE networks, the minimum output of a generator is nonnegative, whereas the minimum output is $0$ in IEEE networks. For IEEE networks, the generator minimum output is scaled to be $1/3$ of the maximum output, which is similar to the level of minimum output in PEGASE and RTE networks.
\subsubsection{Hourly Load Scenarios.} 
Since network information provides a single nominal load demand, we expand the given load demand to the time horizon considered on an hourly basis for one day, i.e., $T = 24$. We benchmark the hourly demand load of power in the U.S. lower region reported from the \cite{eia2022power} and shape the demand load in each network data to create one $T=24$ hourly demand load at each demand bus for a demand load profile in one day. Specifically, suppose that the benchmark demand is denoted as $\bm D^{0}\in \Re_+^{T}$. Let $D_{i} \in \Re_+$ denote the nominal load demand at a demand node $i \in \N$. In order to have the optimal solutions with nontrivial battery operations, we rescale $D_i$ so that $D_i \approx 0.8 G^{\max}_i$. Then, the demand at time $t \in \T$ for node $i \in \N$ is given by $D_{t,i} = D_{i} D^{0}_{t}/D^{0}_{1}$. When running multiple simulations, we add a Gaussian noise with a standard deviation, which is a certain fraction $\hat{\sigma}$ of the demand, to obtain different demand profiles for each simulation. Formally, that is, with some finite number of demand scenarios considered, $\tilde{D}^{j}_{t,i} = D_{t,i} + r_{t,i,j}$ where $r_{t,i,j}\thicksim \mathcal{N}(0,\sigma^2_{t,i})$ and $\sigma_{t,i} = \hat{\sigma} \cdot D_{t,i}$ for simulation $j$.
\begin{figure}[hbt!]
\centering
\includegraphics[width=0.3\textwidth]{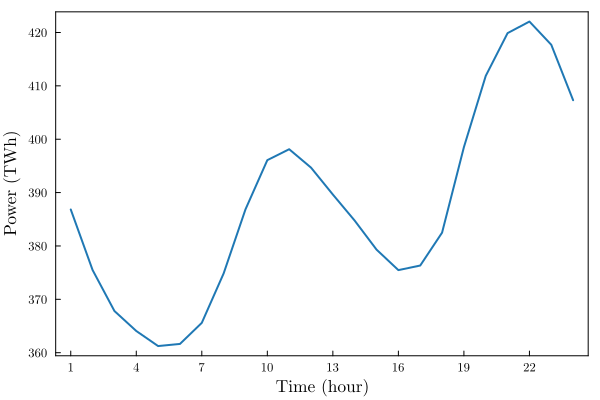}
    \caption{Benchmark Load Demand $\bm D^0$.}
    \label{fig:benchmark_demand}
    \vspace{-0.5cm}
\end{figure}

\subsubsection{Battery Parameters.}
Battery parameters are largely adopted from \cite{kody2022} and modified to account for the size of the network. Table~\ref{tab:battery_parameters} summarizes the battery parameters used for networks including IEEE 73, PEGASE 89, IEEE 118, and IEEE 162-bus systems. For smaller networks,  the maximum storage limit, charge rate, and discharge rate are divided by $5$ for the IEEE 14-bus system. For a slightly larger network IEEE 300-bus system, maximum limit/rates are multiplied by $2.5$. Similarly, for large networks including PEGASE 1354 and RTE 1888, multiplication is by $5$. Unless otherwise stated, these standard battery parameters are used for all numerical experiments in the paper.
\begin{table}[hbt!]
\centering
\small
\setlength\extrarowheight{-2pt}
\caption{Battery Parameters for a Medium-Size Network}
\label{tab:battery_parameters}
\begin{tabular}{lcr}
\toprule
& Parameter & \multicolumn{1}{c}{Value} \\
\midrule
Minimum storage limit  &  $E^{\min}$ &  0.00 p.u. \\
Maximum storage limit  &  $E^{\max}$  &  1.00 p.u.     \\
Efficiency      &  $\eta$ &  0.95 \\
Minimum charge rate    &  $E_c^{\min}$ & 0.00 p.u./hour \\
Maximum charge rate    &  $E_c^{\max}$ & 0.95 p.u./hour \\
Minimum discharge rate &  $E_d^{\min}$ & 0.00 p.u./hour \\
Maximum discharge rate &  $E_d^{\max}$ & 0.95 p.u./hour \\
\bottomrule
\end{tabular}
\end{table}

\subsection{Testing the Regularized MIP Model on DCOPF with ESS}
\label{sec:testing_regularized_mip}
In this section, we show the computational results of applying the regularized MIP model to the DCOPF problem with ESS. 
We solve \eqref{formulation:original_mip} to find the true optimal solution and evaluate the solution from the regularized MIP model, i.e., \eqref{formulation:regularization}. 
Table~\ref{tab:example6_modelcompare} shows the average optimality gap computed as $(\hat z - z^*)/z^*$ where $\hat z$ is the objective value of the feasible solution from \eqref{formulation:regularization} and $z^*$ is the optimal objective value of \eqref{formulation:original_mip}.
The problem is solved with $40$ random demand scenarios on networks. In this setting, $b$ number of batteries are placed on selected buses with the largest power outputs.
Since the actual efficiencies of grid-scale batteries are known to be higher than $80\%$, we consider different efficiency levels $\eta\in \{0.85, 0.9, 0.95\}$. This computational result demonstrates that the regularized MIP yields solutions close to the optimal solution. The average gap across networks is limited to less than $0.5\%$ for different efficiency levels considered.

\begin{table}[hbt!]
\centering
\small
\setlength\extrarowheight{-2pt}
\caption{The average optimality gap with respect to the optimal objective value}
\label{tab:example6_modelcompare}
\begin{tabular}{lrrrr}
\toprule
& & \multicolumn{3}{c}{Efficiency Level} \\
Network   & $b$   & $\eta=0.85$       & $\eta=0.90 $      & $\eta=0.95 $ \\
\midrule
IEEE 14     & 2 & 0.00\% & 0.00\% & 0.00\% \\
IEEE 73     & 2 & 0.33\% & 0.23\% & 0.11\% \\
PEGASE 89   & 2 & 0.00\% & 0.00\% & 0.00\% \\
IEEE 118    & 2 & 0.00\% & 0.00\% & 0.00\% \\
IEEE 162    & 3 & 0.00\% & 0.00\% & 0.00\% \\
IEEE 300    & 3 & 0.08\% & 0.05\% & 0.02\% \\
PEGASE 1354 & 5 & 0.00\% & 0.00\% & 0.00\% \\
RTE 1888    & 5 & 0.05\% & 0.02\% & 0.03\% \\
\midrule
Average     & & 0.06\% & 0.04\% & 0.02\%  \\
\bottomrule
\end{tabular}
\end{table}


\subsection{Computational Results on \texorpdfstring{$N-k$}{Lg} Contingency Problem with ESS Siting}
\label{sec:computation_results_n-k}
We note that the solution methodology from \Cref{sec:trilevel_solution_methodology} to solve trilevel $\min$-$\max$-$\min$ problems \eqref{eq:trilevel_reg} \rev{--} \eqref{eq:trilevel_lp} may not return the optimal solutions within a specified time limit.
In such cases, we use the \rev{maximum of: (i) }lower bound of $z^{LP}$\rev{; or (ii) lower bound of $z^{VCH}$} and the upper bound of $z^{REG}$ to estimate the optimality-gap as the algorithm returns the upper bound and lower bound at each iteration.
Specifically, let $\overline{z^{LP}}$ and $\underline{z^{LP}}$ be the upper and lower bounds of $z^{LP}$ respectively, \rev{$\overline{z^{VCH}}$ and $\underline{z^{VCH}}$ be the upper and lower bounds of $z^{VCH}$ respectively,} and $\overline{z^{REG}}$ and $\underline{z^{REG}}$ be the upper and lower bounds of $z^{REG}$ respectively. Then, we have the following inequality:
\begin{align}
\label{eq:trilevel_solution_quality}
    \text{optimality-gap} = \frac{|z^{UB}-z^{LB}|}{|z^{UB}|} & \leq \frac{|z^{REG}-\rev{\max\{z^{LP}, z^{VCH}\}}|}{|z^{REG}|} \\
    & \leq \frac{|\overline{z^{REG}}-\rev{\max\{\underline{z^{LP}}, \underline{z^{VCH}}\}}|}{|\overline{z^{REG}}|} =:\text{solution-gap}
\end{align}
where $z^{UB}$ and $z^{LB}$ are the upper and lower bounds of the optimal objective value $z^{OPT}$ respectively.
The first inequality follows from Proposition~\ref{proposition:trilevel_max_optimality_gap}.
Let us define the right-hand side in \eqref{eq:trilevel_solution_quality} as a solution-gap.
The criteria for the solution algorithm to terminate is either when the gap between the upper bound (e.g., $\overline{z^{REG}}$) and the lower bound (e.g., $\underline{z^{REG}}$) is less than the specified limit of 0.5\% or the algorithm runs for a time limit of 6 hours (i.e., 21600 seconds).
In Table~\ref{tab:trilevel_result}, we report the minimum, maximum, and average of the solution-gap from $10$ simulations for each network instance and different combinations of the maximum number of batteries placed and the maximum number of contingencies. 

\begin{table}[hbt!]
\caption{Solution-Gap for $N-k$ Contingency Problem}
\label{tab:trilevel_result}
\centering
\small
\setlength{\tabcolsep}{2pt}
\setlength\extrarowheight{-2pt}
\begin{tabular}{lrrrrrrrrrrrrrrr}
\toprule
& \multicolumn{15}{c}{Optimality-Gap}                              \\
& \multicolumn{3}{c}{$b = 2, k = 3$} &  & \multicolumn{3}{c}{$b = 2, k = 5$} &  & \multicolumn{3}{c}{$b = 3, k = 5$} &  & \multicolumn{3}{c}{$b = 5, k = 10$} \\ \cline{2-4} \cline{6-8} \cline{10-12} \cline{14-16}
Network    & Min       & Max       & Avg      &  & Min       & Max       & Avg      &  & Min       & Max       & Avg      &  & Min       & Max       & Avg       \\
\midrule
IEEE 14     & 0.62\% & 0.63\%  & 0.62\%  &  & 0.57\% & 0.58\%  & 0.57\%  &  & 0.57\%  & 0.58\%  & 0.57\%  &  & 0.57\%  & 0.58\%  & 0.57\%  \\
IEEE 73 & 0.00\% & 0.07\% & 0.03\% &  & 0.00\% & \rev{0.04\%} & \rev{0.01\%} &  & 0.00\% & \rev{0.03\%} & \rev{0.01\%} &  & 0.00\% & 0.02\% & 0.01\% \\
PEGASE 89 & \rev{0.12\%} & \rev{0.12\%} & \rev{0.12\%} &  & \rev{0.10\%} & \rev{0.10\%} & \rev{0.10\%} &  & \rev{0.10\%} & \rev{0.18\%} & \rev{0.16\%} &  & 0.27\% & 1.16\% & 1.04\% \\
IEEE 118 & \rev{0.88\%} & \rev{1.07\%} & \rev{0.95\%} &  & \rev{1.12\%} & \rev{1.73\%} & \rev{1.42\%} &  & \rev{1.81\%} & \rev{2.88\%} & \rev{2.33\%} &  & 2.07\% & 4.78\% & 3.38\% \\
IEEE 162 & \rev{0.20\%} & \rev{0.20\%} & \rev{0.20\%} &  & \rev{0.18\%} & \rev{0.18\%} & \rev{0.18\%} &  & \rev{0.12\%} & \rev{2.77\%} & \rev{1.32\%} &  & \rev{0.34\%} & \rev{2.79\%} & \rev{2.22\%} \\
IEEE 300 & 3.69\% & 4.08\% & 3.81\% &  & 3.41\% & \rev{3.61\%} & 3.48\% &  & 3.69\% & \rev{4.03\%} & \rev{3.82\%} &  & 6.21\% & \rev{8.19\%} & \rev{6.96\%} \\
PEGASE 1354 & 9.81\% & \rev{10.15\%} & \rev{10.06\%} &  & \rev{7.83\%} & \rev{8.44\%} & \rev{8.03\%} &  & 15.00\% & \rev{16.10\%} & \rev{15.53\%} &  & 16.98\% & \rev{20.04\%} & \rev{18.21\%} \\
RTE 1888 & 6.73\% & \rev{7.21\%} & \rev{6.95\%} &  & 6.32\% & \rev{7.14\%} & \rev{6.69\%} &  & 12.11\% & \rev{14.48\%} & \rev{12.90\%} &  & 17.66\% & \rev{26.72\%} & \rev{22.83\%} \\
\bottomrule
\end{tabular}
\vspace{-0.3cm}
\end{table}

We observe that for smaller network systems, the regularized MIP generates provably near-optimal solutions. 
The fact that the solution-gap is $0.00$\% implies that the optimality-gap is $0.00\%$ and $z^{REG}=\rev{\max\{z^{LP}, z^{VCH}\}}=z^{OPT}$ implying that both solving with the regularized MIP model and solving with \rev{one of} the LP relaxation model\rev{s} yield the true optimal solution \rev{to} the trilevel problem.
For PEGASE 1354 and RTE 1888-bus systems, the solution-gap is larger, but this is most likely due to having a poor lower bound on \rev{either} $z^{LP}$ \rev{or $z^{VCH}$} as $\overline{z^{LP}}$ and $\underline{z^{LP}}$ do not converge in \eqref{eq:trilevel_lp} \rev{or $\overline{z^{VCH}}$ and $\underline{z^{VCH}}$ do not converge in \eqref{eq:trilevel_vch}} within the given time limit of 6 hours. \rev{Generally, the VCH model was found to perform better than the LP relaxation especially for larger networks. 
For most of the instances when we see VCH has larger gap than LP, both VCH and LP have relative gaps less than 0.5\%, which is one of the termination criteria of the bilevel algorithm. }

Let us define another measure, trilevel-gap, to indicate the gap we obtain from the algorithm in solving the trilevel problem of the form \eqref{eq:trilevel_reg}, \eqref{eq:trilevel_lp}, \rev{or \eqref{eq:trilevel_vch}}. For example, the trilevel-gap of \eqref{eq:trilevel_reg} is:
\begin{align*}
    \text{trilevel-gap of \eqref{eq:trilevel_reg}} = \frac{|\overline{z^{REG}}-\underline{z^{REG}}|}{|\overline{z^{REG}}|}.
\end{align*}
These trilevel-gaps in solving \eqref{eq:trilevel_reg} \rev{--} \rev{\eqref{eq:trilevel_vch}} are reported in Tables~\ref{tab:trilevel_solve_time_opt_gap_b2_k3}-\ref{tab:trilevel_solve_time_opt_gap_b5_k10} where we see that the averages of trilevel-gap\rev{s} of \eqref{eq:trilevel_lp} \rev{and \eqref{eq:trilevel_vch}} for PEGASE 1354-bus system \rev{can reach beyond 18\% and for RTE 1888-bus system 22\%}.

There are also cases where both algorithms using the regularized MIP model and the LP relaxation model achieve $0.00\%$ trilevel-gaps but result in the solution-gap that is greater than $0.00\%$, which implies that there is a true gap in approximating the true model with either the regularized MIP model or the LP relaxation model. 
This is the case for some of the smaller network systems like the PEGASE 89-bus system.
Nevertheless, for these smaller systems, the solution-gap remains relatively limited.
For example, for the PEGASE 89-bus system with $b=2$ and $k=3$ case, the average solution-gap is $0.16\%$ which implies that the optimal objective value of the regularized MIP overestimates the optimal objective value of the original MIP model by no more than $0.16\%$.

We observe that solving the trilevel problem with the regularized formulation is efficient. 
For most of the instances, \eqref{eq:trilevel_reg} is solved to optimality well within the time limit so that the trilevel-gap is $0.00\%$. 
Only the larger network instances take the entire 6-hour time limit with a small trilevel-gap. 
Hence, using the regularized formulation not only gives a quality upper bound but can be used to solve such trilevel problems efficiently.

\begin{table}[hbt!]
\caption{Average Solution Time and Trilevel-Gap ($b=2$ and $k=3$)}
\label{tab:trilevel_solve_time_opt_gap_b2_k3}
\centering
\small
\setlength\extrarowheight{-2pt}
\begin{tabular}{lrrrrrrrr}
\toprule
& \multicolumn{2}{c}{$z^{REG}$} & & \multicolumn{2}{c}{$z^{LP}$} & & \multicolumn{2}{c}{{$z^{VCH}$}} \\ \cline{2-3} \cline{5-6} \cline{8-9}
\multicolumn{1}{l}{Network} & \multicolumn{1}{c}{Time (sec)} & \multicolumn{1}{c}{Trilevel-Gap} & \multicolumn{1}{c}{} & \multicolumn{1}{c}{Time (sec)} & \multicolumn{1}{c}{Trilevel-Gap} & \multicolumn{1}{c}{{}} & \multicolumn{1}{c}{\rev{Time (sec)}} & \multicolumn{1}{c}{\rev{Trilevel-Gap}} \\
\midrule
IEEE 14 & 8 & 0.00\% &  & 7 & 0.00\% & & \rev{7} & \rev{0.00\%}  \\
IEEE 73 & 479 & 0.00\% &  & 522 & 0.00\% & & \rev{763} & \rev{0.05\%} \\
PEGASE 89 & 12 & 0.00\% &  & 70 & 0.02\% & & \rev{171} & \rev{0.00\%} \\
IEEE 118 & 9970 & 0.00\% &  & 7817 & 0.19\% & & \rev{15852} & \rev{0.07\%} \\
IEEE 162 & 14 & 0.00\% &  & 137 & 0.00\% & & \rev{394} & \rev{0.00\%} \\
IEEE 300 & 63 & 0.13\% &  & 3238 & 3.33\% & & \rev{5564} & \rev{3.52\%} \\
PEGASE 1354 & 20272 & 2.21\% &  & 19388 & 13.25\% & & \rev{21600} & \rev{9.06\%} \\
RTE 1888 & 8961 & 0.14\% &  & 22304 & 7.46\% & & \rev{22517} & \rev{6.77\%} \\
\bottomrule
\end{tabular}
\vspace{-1em}
\end{table}

\begin{table}[hbt!]
\caption{Average Solution Time and Trilevel-Gap ($b=2$ and $k=5$)}
\label{tab:trilevel_solve_time_opt_gap_b2_k5}
\centering
\small
\setlength\extrarowheight{-2pt}
\begin{tabular}{lrrrrrrrr}
\toprule
& \multicolumn{2}{c}{$z^{REG}$} & & \multicolumn{2}{c}{$z^{LP}$} & & \multicolumn{2}{c}{{$z^{VCH}$}} \\ \cline{2-3} \cline{5-6} \cline{8-9}
\multicolumn{1}{l}{Network} & \multicolumn{1}{c}{Time (sec)} & \multicolumn{1}{c}{Trilevel-Gap} & \multicolumn{1}{c}{} & \multicolumn{1}{c}{Time (sec)} & \multicolumn{1}{c}{Trilevel-Gap} & \multicolumn{1}{c}{{}} & \multicolumn{1}{c}{\rev{Time (sec)}} & \multicolumn{1}{c}{\rev{Trilevel-Gap}} \\
\midrule
IEEE 14 & 8 & 0.00\% &  & 7 & 0.00\% &  & {8} & {0.00\%}\\
IEEE 73 & 601 & 0.00\% &  & 697 & 0.05\% &  & {843} & {0.12\%} \\
PEGASE 89 & 16 & 0.00\% &  & 113 & 0.00\% &  & {273} & {0.00\%} \\
IEEE 118 & 3752 & 0.67\% &  & 2804 & 0.25\% &  & {5005} & {0.37\%} \\
IEEE 162 & 34 & 0.00\% &  & 550 & 0.00\% &  & {473} & {0.33\%} \\
IEEE 300 & 92 & 0.07\% &  & 3618 & 3.11\% &  & {5113} & {3.25\%} \\
PEGASE 1354 & 14453 & 1.83\% &  & 21600 & 13.04\% &  & {21600} & {7.32\%} \\
RTE 1888 & 12267 & 0.67\% &  & 21600 & 7.65\% &  & {21600} & {6.63\%} \\
\bottomrule
\end{tabular}
\vspace{-1em}
\end{table}

\begin{table}[hbt!]
\caption{Average Solution Time and Trilevel-Gap ($b=3$ and $k=5$)}
\label{tab:trilevel_solve_time_opt_gap_b3_k5}
\centering
\small
\setlength\extrarowheight{-2pt}
\begin{tabular}{lrrrrrrrr}
\toprule
& \multicolumn{2}{c}{$z^{REG}$} & & \multicolumn{2}{c}{$z^{LP}$} & & \multicolumn{2}{c}{{$z^{VCH}$}} \\ \cline{2-3} \cline{5-6} \cline{8-9}
\multicolumn{1}{l}{Network} & \multicolumn{1}{c}{Time (sec)} & \multicolumn{1}{c}{Trilevel-Gap} & \multicolumn{1}{c}{} & \multicolumn{1}{c}{Time (sec)} & \multicolumn{1}{c}{Trilevel-Gap} & \multicolumn{1}{c}{{}} & \multicolumn{1}{c}{\rev{Time (sec)}} & \multicolumn{1}{c}{\rev{Trilevel-Gap}} \\
\midrule
IEEE 14 & 7 & 0.00\% &  & 8 & 0.00\% &  & {6} & {0.00\%}\\
IEEE 73 & 993 & 0.00\% &  & 730 & 0.05\% &  & {1006} & {0.20\%} \\
PEGASE 89 & 16 & 0.00\% &  & 245 & 0.00\% &  & {669} & {0.06\%} \\
IEEE 118 & 4441 & 0.98\% &  & 2205 & 0.22\% &  & {8857} & {0.75\%} \\
IEEE 162 & 32 & 0.04\% &  & 1325 & 2.24\% &  & {1384} & {1.29\%} \\
IEEE 300 & 192 & 0.12\% &  & 3356 & 3.53\% &  & {7749} & {3.51\%} \\
PEGASE 1354 & 21600 & 3.89\% &  & 21600 & 17.70\% &  & {21600} & {14.82\%} \\
RTE 1888 & 14262 & 0.24\% &  & 21600 & 14.19\% &  & {21600} & {12.84\%} \\
\bottomrule
\end{tabular}
\vspace{-1em}
\end{table}

\vspace{-0.5em}
\begin{table}[hbt!]
\caption{Average Solution Time and Trilevel-Gap ($b=5$ and $k=10$)}
\label{tab:trilevel_solve_time_opt_gap_b5_k10}
\centering
\small
\setlength\extrarowheight{-2pt}
\begin{tabular}{lrrrrrrrr}
\toprule
& \multicolumn{2}{c}{$z^{REG}$} & & \multicolumn{2}{c}{$z^{LP}$} & & \multicolumn{2}{c}{{$z^{VCH}$}} \\ \cline{2-3} \cline{5-6} \cline{8-9}
\multicolumn{1}{l}{Network} & \multicolumn{1}{c}{Time (sec)} & \multicolumn{1}{c}{Trilevel-Gap} & \multicolumn{1}{c}{} & \multicolumn{1}{c}{Time (sec)} & \multicolumn{1}{c}{Trilevel-Gap} & \multicolumn{1}{c}{{}} & \multicolumn{1}{c}{\rev{Time (sec)}} & \multicolumn{1}{c}{\rev{Trilevel-Gap}} \\
\midrule
IEEE 14 & 9 & 0.00\% &  & 9 & 0.00\% &  & {10} & {0.00\%}\\
IEEE 73 & 1155 & 0.00\% &  & 1010 & 0.00\% &  & {1410} & {0.35\%} \\
PEGASE 89 & 21 & 0.01\% &  & 374 & 0.80\% &  & {1354} & {0.96\%} \\
IEEE 118 & 16742 & 0.25\% &  & 19982 & 1.22\% &  & {22267} & {4.67\%} \\
IEEE 162 & 16 & 0.21\% &  & 523 & 1.94\% &  & {2032} & {1.63\%} \\
IEEE 300 & 22 & 0.09\% &  & 1241 & 6.96\% &  & {4121} & {6.53\%} \\
PEGASE 1354 & 12653 & 1.54\% &  & 21600 & 18.14\% &  & {21600} & {18.46\%} \\
RTE 1888 & 21600 & 1.29\% &  & 21600 & 24.97\% &  & {21600} & {22.70\%} \\
\bottomrule
\end{tabular}
\vspace{-1em}
\end{table}

\section{Conclusion}
\label{sec:conclusion}
In this paper, we proposed a new model to solve the DCOPF problem with battery operations. We regularized the objective function by penalizing the charge and discharge of batteries. In Theorem~\ref{theorem:ip=lp_condition}, we present a sufficient condition on the regularizers so that there is no integrality gap between the regularized MIP problem and its LP relaxation.
When the efficiency of the battery is relatively high, this penalty is very small. 
Empirical results show that the optimal solution from this regularized model is often a true optimal solution to the original model or close to the optimal solution, performing much better than the theoretical guarantees verified in Theorem~\ref{theorem:worst_bound}. Moreover, we prove in Theorem~\ref{thm:structure} that the optimal solution from the regularized model is more reasonable in that the battery operation does not contribute to further load-shedding or excess power depending on the state of the system at the time. This property may be of interest to the system operator. \rev{We note that these properties depend on the structure of the objective function and the constraint. It will be interesting to study the structure with respect to different objective functions, which we leave to future works.}

For a simpler problem that only considers the battery operation with \rev{$\eta_c = \eta_d =1$}, a polynomial algorithm has been proposed in \cite{polyalgo2023ostrowski}. 
However, only a few studies have focused on the complexity associated with the general efficiency level with \rev{$0<\eta_c, \eta_d < 1$}. \cite{bansal2023warehouse} proved an NP-hardness for a similar problem where the storage level varies over time based on two complementary variables.
The proof of NP-hardness relies on time-varying bounds and their result shows that \eqref{formulation:original_mip} problem with time-varying bounds on $\bm p^c$ and $\bm p^d$ is NP-hard to solve.
However, it remains an open problem that for a fixed bound on charge and discharge levels with a loss-incurring battery system (\rev{$0<\eta_c, \eta_d < 1$}), the problem is NP-hard.

We introduce a long-term planning problem that includes the battery siting problem with $N-k$ contingency.
This problem is intractable to solve using the exact battery formulation. We use the main benefit of the regularized formulation model to reformulate this challenging problem and show that the regularized formulation solves large-scale instances of these problems efficiently and yields near-optimal solutions in most cases.

\section*{Acknowledgment} 
The authors would like to thank Mathieu Dahan for his valuable suggestions on the earlier version of this work. \rev{Valuable comments from the associate editor and two anonymous reviewers are gratefully acknowledged}.

\bibliography{reference_20250109}

\newpage

\begin{appendices}
\section{Proof of Theorem \ref{thm:structure}}
\label{proof_thmstructure}
\thmstructure*
The proof of Theorem~\ref{thm:structure} is divided into two parts via Lemma~\ref{lemma:pc_pls=0} and Lemma~\ref{lemma:pd_pex=0}.

\begin{restatable}{lemma}{lemmapcpls0}\label{lemma:pc_pls=0}
Suppose $E^{\min}_c = E^{\min}_d = 0$. For any $\bm \lambda \in \mathbb{R}^2_+$, let $\bm p$ be an optimal solution to \eqref{formulation:regularization} problem. Then, $p^c_{t,i} p^{ls}_{t,i} = 0$ for all $ t \in \T,\;i \in \N$.
\end{restatable}

\begin{proof}
By contradiction, suppose there exists an optimal solution $(\hat{\bm \theta}, \hat{\bm f}, \hat{\bm p}, \hat{\bm u})$ of \eqref{formulation:regularization} problem such that $\hat{p}^c_{t,i} > 0$ and  $\hat{p}^{ls}_{t,i} > 0$ for at least one $(t,i)\in \T\times \N$.

Without loss of generality, we show the proof for one such $i \in \N$ as the proof can be extended for any multiple nodes. Let $\tau_0 \in \T$ be the first time period such that $\hat{p}^c_{\tau_0,i} > 0$ and $\hat{p}^{ls}_{\tau_0,i} > 0$. Let $\tau_1, \dots, \tau_k \in \{\tau_0 + 1, \dots, T \}$ such that $\hat{p}^d_{\tau_j, i} > 0$ for $j\in [k]$ and $\hat{p}^d_{t.i} = 0$ for $t \in \{\tau_0 +1, \dots, T \} \setminus \{\tau_1, \dots, \tau_k\}$. We define adjustments to the state-of-charge as the following:
\begin{align*}
    \delta_{\tau_j,i}  &= 
    \left\{\begin{array}{ll} 
    -\rev{\eta_c}\cdot\min \{\hat{p}^c_{\tau_0,i},\hat{p}^{ls}_{\tau_0,i}\}, &\quad  j = 0, \\
    \delta_{\tau_{j -1}, i} + \max \left\{E^{\min} - \hat{p}^s_{\tau_j, i} - \delta_{\tau_{j -1}, i}, 0 \right\}, &\quad \forall j \in [k].
    \end{array}\right.
\end{align*}
Note that $\delta_{\tau_0,i} < 0$ \rev{and $\delta_{\tau_j,i}\leq 0$ for all $j \in [k]$}. We proceed to construct a solution $(\tilde{\bm \theta}, \tilde{\bm f}, \tilde{\bm p}, \tilde{\bm u})$ from the current optimal solution $(\hat{\bm \theta}, \hat{\bm f}, \hat{\bm p}, \hat{\bm u})$, where $\tilde{\bm \theta} = \hat{\bm \theta}$, $\tilde{\bm f} =  \hat{\bm f}$, $\tilde{\bm u}=  \hat{\bm u}$, $\tilde{\bm p}^{ex} = \hat{\bm p}^{ex}$, $\tilde{\bm p}^g = \hat{\bm p}^g,$  and changing values only corresponding to node $i$ as follows:
\begin{align*}
\tilde{p}^c_{t,i} &= \left\{\begin{array}{ll} \hat{p}^c_{t,i}, &\quad  \forall t\in \T\setminus \{\tau_0\}, \\
\hat{p}^c_{\tau_0,i} - \min \{\hat{p}^c_{\tau_0,i},\; \hat{p}^{ls}_{\tau_0,i}\}, & \quad t = \tau_0,
\end{array}\right.\\
\tilde{p}^d_{t,i} &= \left\{\begin{array}{ll} \hat{p}^d_{t,i}, &\quad  \forall t \in \T \setminus \{\tau_1, \dots, \tau_k\}, \\
\hat{p}^d_{t,i} - \rev{\eta_d}\cdot \max \left\{E^{\min} - \hat{p}^s_{\tau_j, i} - \delta_{\tau_{j -1}, i}, 0 \right\}, & \quad t = \tau_j, \ \forall j\in [k],
\end{array}\right.\\
\tilde{p}^{ls}_{t,i} &= \left\{\begin{array}{ll} \hat{p}^{ls}_{t,i}, &\quad  \forall t \in \T \setminus \{\tau_0, \tau_1, \dots, \tau_k\}, \\
\hat{p}^{ls}_{t,i} - \min \{\hat{p}^c_{\tau_0,i},\; \hat{p}^{ls}_{\tau_0,i}\},  & \quad t = \tau_0,\\
\hat{p}^{ls}_{t,i} + \rev{\eta_d}\cdot \max \left\{E^{\min} - \hat{p}^s_{\tau_j, i} - \delta_{\tau_{j -1}, i}, 0 \right\}, &\quad  t = \tau_j, \ \forall j\in [k],
\end{array}\right.\\
\tilde{p}^{s}_{t,i} &= \left\{\begin{array}{ll} \hat{p}^s_{t,i}, &\quad  \forall t \in [\tau_0-1], \\
\hat{p}^{s}_{t,i} + \delta_{\tau_j,i}, &\quad  \forall t \in \{\tau_j, \dots, \tau_{j +1}-1\}, \; \forall j \in \{0, \dots, k\},\\
\end{array}\right.
\end{align*}
where $\tau_{k +1}-1 = T.$

The new solution $\tilde{\bm p}$ above is created in the following fashion: We first reduce both $\hat{p}^c_{\tau_0,i}$ and $\hat{p}^{ls}_{\tau_0, i}$, which causes $\hat{p}^s_{t, i}$ to reduce in time periods following $\tau_0$. In particular, it may fall below $E^{\min}$. In order to fix this, we need to modify the discharging levels (and loss values of corresponding time periods) to ensure that the storage levels meet the minimum requirement $E^{\min}$. We carefully decrease the values of $\hat{p}^{d}_{t,i}$ so that the minimum discharge level is satisfied and the state-of-charge level is never below $E^{\min}$ at the same time.

\begin{restatable}{claim}{claim_pcpls_newp_feasibility}\label{claim:pcpls_newp_feasibility}
$(\tilde{\bm \theta}, \tilde{\bm f}, \tilde{\bm p}, \tilde{\bm u})$ is a feasible solution to \eqref{formulation:regularization} with given $\bm \lambda$.
\end{restatable}
\begin{proof}
It suffices to show that $\tilde{\bm p}$ satisfies $\eqref{eq:battery_pc_limits} - \eqref{eq:battery_storage}$, \eqref{eq:battery_ps_limits} and $\eqref{eq:battery_plspex_lb} - \eqref{eq:power_balance}$.
\begin{itemize}
    \item \eqref{eq:battery_pc_limits}: It is straightforward to verify that $\tilde{p}^c_{t,i} \geq 0 = E^{\min}_c$ for all $t \in \T$.
    \item \eqref{eq:battery_pd_limits}: We want to show that $\tilde{p}^{d}_{t,i} \geq 0 = E^{\min}_d$ for all $t \in \T$. It is sufficient to prove this for $t = \{ \tau_1, \dots, \tau_k\}$. There are two cases:
    \begin{itemize}
        \item[(i)] If $E^{\min} - \hat{p}^s_{\tau_j, i} - \delta_{\tau_{j -1}, i} \leq 0$: This case is straightforward as $\tilde{p}^d_{t,i} = \hat{p}^d_{t,i} \geq 0.$
        \item[(ii)] If $E^{\min} - \hat{p}^s_{\tau_j, i} - \delta_{\tau_{j -1}, i} > 0$: In this case, note that 
        \rev{$\tilde{p}^d_{t,i} = \eta_d (\tilde{p}^s_{t-1,i} - \tilde{p}^s_{t,i}) > \eta_d (\tilde{p}^s_{t-1,i} + \delta_{\tau_{j-1},i} - E^{\text{min}}) = \eta_d (\hat{p}^s_{t-1,i} - E^{\text{min}})\geq 0$, where the first equality comes from \eqref{eq:battery_storage}.}
    \end{itemize}
    \item \eqref{eq:battery_storage}: It is straightforward to verify that $\tilde{p}^s_{t,i} = \tilde{p}^s_{t-1,i} + \rev{\eta_c}\cdot \tilde{p}^c_{t,i} - {1}/{\rev{\eta_d}}\cdot\tilde{p}^d_{t,i}$ for all $t \in \T$.
    \item \eqref{eq:battery_ps_limits}: We want to show that $\tilde{p}^s_{t,i} \geq E^{\min}$ for all $t \in \T$. Clearly $\tilde{p}^s_{t,i} = \hat{p}^s_{t,i} \geq E^{\min}$ for all $t \in [\tau_0-1]$. For $t \geq \tau_0$, we show this in three parts: 
    \begin{itemize}
        \item[(i)] For $t = \tau_0$, from \eqref{eq:battery_storage} and the constructions above, we have $\hat{p}^s_{\tau_0, i} = \hat{p}^s_{\tau_0-1, i} + \rev{\eta_c}\cdot \hat{p}^c_{\tau_0, i} \geq E^{\min} + \rev{\eta_c} \cdot \min\{\hat{p}^c_{\tau_0,i},\; \hat{p}^{ls}_{\tau_0,i}\}.$ Therefore, $\tilde{p}^s_{\tau_0, i} = \hat{p}^s_{\tau_0, i} - \rev{\eta_c} \cdot \min\{\hat{p}^c_{\tau_0,i},\; \hat{p}^{ls}_{\tau_0,i}\} \geq E^{\min}.$
        \item[(ii)] For $t = \tau_j$ for all $j \in [k]$, we have $\tilde{p}^s_{\tau_j, i}  = \hat{p}^s_{\tau_j,i} + \delta_{\tau_j, i} = \hat{p}^s_{\tau_j,i}  + \delta_{\tau_{j -1}, i} + \max\{E^{\min} - \hat{p}^s_{\tau_j,i} - \delta_{\tau_{j -1}, i}, 0 \} \geq E^{\min}$. 
        \item[(iii)] Finally, for any $t \in \{\tau_j+1, \dots, \tau_{j +1}-1\}$ for $j \in \{0, \dots, k\}$, observe that since $\hat{p}^d_{\tau_j +1, i} = \dots = \hat{p}^d_{\tau_{j +1} - 1, i} = 0,$ we have $\tilde{p}^s_{t, i} = \hat{p}^s_{t,i} + \delta_{\tau_j, i} \geq \hat{p}^s_{\tau_j, i} + \delta_{\tau_j, i} = \tilde{p}_{\tau_j, i} \geq E^{\min}$, where the last inequality follows from the above.
    \end{itemize}
    \item \eqref{eq:battery_plspex_lb}: It is straightforward to verify that $\tilde{p}^{ls}_{t,i} \geq 0$ for all $t \in \T$.
    \item \eqref{eq:power_balance}: It is straightforward to verify that $\hat{p}^c_{t,i} + \hat{p}^{ex}_{t,i} - \hat{p}^{d}_{t,i} - \hat{p}^{ls}_{t,i} = \tilde{p}^c_{t,i} + \tilde{p}^{ex}_{t,i} - \tilde{p}^{d}_{t,i} - \tilde{p}^{ls}_{t,i}$ for all $t \in \T$.
\end{itemize}
\QEDA
\end{proof}
Observe that $\tilde{p}^c_{\tau_0,i}\cdot \tilde{p}^{ls}_{\tau_0,i} = 0$.
Next, we claim that the total additional adjustment made to $\delta$ over time is upper bounded.
\begin{restatable}{claim}{claim_pcpls_sum_ub}\label{claim:pcpls_sum_ub}
$\sum_{j = 1}^k\max\{E^{\min} - \hat{p}^s_{\tau_j, i} - \delta_{\tau_{j -1}, i}, 0 \} \leq |\delta_{\tau_0, i}|$.
\end{restatable}
\begin{proof}
Suppose that $S \subseteq [k]$ such that $\max\{E^{\min} - \hat{p}^s_{\tau_j, i} - \delta_{\tau_{j -1}, i}, 0 \} = E^{\min} - \hat{p}^s_{\tau_j, i} - \delta_{\tau_{j -1},i}$. Let $S = \{j_1, j_2, \dots, j_m\}$. If $S = \emptyset$, then there is nothing to verify. Otherwise, it is straightforward to verify that:
\begin{align*}
    & \sum_{j = 1}^k\max\{E^{\min} - \hat{p}^s_{\tau_j, i} - \delta_{\tau_{j}-1, i}, 0 \} \\
    & = (E^{\min} - \hat{p}^s_{\tau_{j_1}, i} - \delta_{\tau_{0}, i}) + (E^{\min} - \hat{p}^s_{\tau_{j_2}, i} - \delta_{\tau_{j_2}-1, i}) + \cdots + (E^{\min} - \hat{p}^s_{\tau_{j_m}, i} - \delta_{\tau_{j_{m}-1}, i}) \\
    & = (E^{\min} - \hat{p}^s_{\tau_{j_1}, i} - \delta_{\tau_{0}, i}) + (\hat{p}^s_{\tau_{j_1}, i} - \hat{p}^s_{\tau_{j_2}, i}) + \cdots + (\hat{p}^s_{\tau_{j_{m-1}}, i} - \hat{p}^s_{\tau_{j_m}, i}) \\
    & = E^{\min} - \hat{p}^s_{\tau_{j_m}, i} - \delta_{\tau_{0}, i} \leq |\delta_{\tau_{0}, i}|.
\end{align*}
The second equality is due to the fact that $\delta_{\tau_{j_l},i} = E^{\min} - \hat{p}^s_{\tau_{j_l},i}$ and $\delta_{\tau_{j_{l+1}}-1,i} = \delta_{\tau_{j_l},i}$ for all $l=1,...,m$.
\QEDA
\end{proof}
Finally, we are ready to compute the difference in objective function value of the two solutions:
\begin{align*}
& c(\tilde{\bm p}) + \bm \lambda^\top g(\tilde{\bm p}) - c(\hat{\bm p}) - \bm \lambda^\top g(\hat{\bm p}) \\
= & \sum_{t \in \T} \sum_{i \in N} \left( \tilde{p}^{ls}_{t,i} - \hat{p}^{ls}_{t,i} + \tilde{p}^{ex}_{t,i} - \hat{p}^{ex}_{t,i} + \lambda_c ( \tilde{p}^{c}_{t,i} - \hat{p}^{c}_{t,i}) + \lambda_d (\tilde{p}^{d}_{t,i} - \hat{p}^{d}_{t,i}) \right)\\
= & - \min \{\hat{p}^c_{\tau_0,i},\; \hat{p}^{ls}_{\tau_0,i}\} + \sum_{j \in [k]} \rev{\eta_d} \cdot \max \left\{E^{\min} - \hat{p}^s_{\tau_j, i} - \delta_{\tau_{j -1}, i}, 0 \right\} \\
& - \lambda_c \cdot \min \{\hat{p}^c_{\tau_0,i},\; \hat{p}^{ls}_{\tau_0,i}\} - \lambda_d \cdot \sum_{j \in [k]} \rev{\eta_d} \cdot \max \left\{E^{\min} - \hat{p}^s_{\tau_j, i} - \delta_{\tau_{j -1}, i}, 0 \right\}\\
= & - (1+\lambda_c) \cdot \frac{|\delta_{\tau_0,i}|}{\rev{\eta_c}} + (1 - \lambda_d) \cdot \rev{\eta_d} \cdot \left(\sum_{j \in [k]} \max\left\{E^{\min} - \hat{p}^s_{\tau_j + 1, i} - \delta_{\tau_{j -1}, i}, 0 \right\}\right) \\
\le & - \frac{|\delta_{\tau_0,i}|}{\rev{\eta_c}} - \lambda_c \frac{|\delta_{\tau_0,i}|}{\rev{\eta_c}} + (1 - \lambda_d) \cdot \rev{\eta_d} \cdot |\delta_{\tau_0,i}| \\
= & \rev{\frac{|\delta_{\tau_0,i}|}{\eta_c} \left( -1 + \eta_c \eta_d - \lambda_c - \eta_c\eta_d \lambda_d \right)} < 0,
\end{align*}
where the last inequality is from \Cref{claim:pcpls_sum_ub}. Therefore,  $c(\tilde{\bm p}) + \bm \lambda^\top g(\tilde{\bm p}) < c(\hat{\bm p}) + \lambda^\top g(\hat{\bm p})$, hence a contradiction. 
\QEDA
\end{proof}

\begin{restatable}{lemma}{lemmapdpex0}\label{lemma:pd_pex=0}
Suppose $E^{\min}_c = E^{\min}_d = 0$. For any $(\lambda_c,\; \lambda_d)$ such that \rev{$\lambda_c + \eta_c \eta_d \cdot \lambda_d > 1-\eta_c \eta_d$} with a given \rev{$\eta_c, \eta_d \in (0,1]$}, let $\bm p$ be an optimal solution to \eqref{formulation:regularization} problem. Then, $p^d_{t,i} p^{ex}_{t,i} = 0$ for all $t \in \T,\;i \in \N$.
\end{restatable}

\begin{proof}
By contradiction, suppose there exists an optimal solution $(\hat{\bm \theta}, \hat{\bm f}, \hat{\bm p}, \hat{\bm u})$ of \eqref{formulation:regularization} problem such that $\hat{p}^d_{t,i} > 0$ and  $\hat{p}^{ex}_{t,i} > 0$ for at least one $(t,i)\in \T\times \N$.

Without loss of generality, we show the proof for one such $i \in \N$ as the proof can be extended for any multiple nodes. Let $\tau_0 \in \T$ be the first time period such that $\hat{p}^d_{\tau_0,i} > 0$ and $\hat{p}^{ex}_{\tau_0,i} > 0$. Let $\tau_1, \dots, \tau_k \in \{\tau_0 + 1, \dots, T \}$ such that $\hat{p}^c_{\tau_j, i} > 0$ for $j\in [k]$ and $\hat{p}^c_{t.i} = 0$ for $t \in \{\tau_0 +1, \dots, T\} \setminus \{\tau_1, \dots, \tau_k\}$. We define adjustments to the state-of-charge as the following:
\begin{align*}
    \delta_{\tau_j,i}  &= 
    \left\{\begin{array}{ll} 
    {1}/{\rev{\eta_d}}\cdot\min \{\hat{p}^d_{\tau_0,i},\; \hat{p}^{ex}_{\tau_0,i}\}, &\quad  j = 0, \\
    \delta_{\tau_{j -1}, i} - \max \left\{\hat{p}^s_{\tau_j, i} + \delta_{\tau_{j -1},i} - E^{\max}, 0 \right\}, &\quad \forall j \in [k].
    \end{array}\right.
\end{align*}
Note that $\delta_{\tau_0,i} < 0$ \rev{and $\delta_{\tau_j,i} \leq 0$ for all $j \in [k]$}. We proceed to construct a solution $(\tilde{\bm \theta}, \tilde{\bm f}, \tilde{\bm p}, \tilde{\bm u})$ from the current optimal solution $(\hat{\bm \theta}, \hat{\bm f}, \hat{\bm p}, \hat{\bm u})$, where $\tilde{\bm \theta} = \hat{\bm \theta}$, $\tilde{\bm f} =  \hat{\bm f}$, $\tilde{\bm u}=  \hat{\bm u}$, $\tilde{\bm p}^{ex} = \hat{\bm p}^{ex}$, $\tilde{\bm p}^g = \hat{\bm p}^g,$  and changing values only corresponding to node $i$ as follows:
\begin{align*}
\tilde{p}^d_{t,i} &= \left\{\begin{array}{ll} \hat{p}^d_{t,i}, &\quad  \forall t\in \T\setminus \{\tau_0\}, \\
\hat{p}^d_{\tau_0,i} - \min \{\hat{p}^d_{\tau_0,i},\; \hat{p}^{ex}_{\tau_0,i}\}, & \quad t = \tau_0,
\end{array}\right.\\
\tilde{p}^c_{t,i} &= \left\{\begin{array}{ll} \hat{p}^c_{t,i}, &\quad  \forall t \in \T \setminus \{\tau_1, \dots, \tau_k\}, \\
\hat{p}^c_{t,i} - \frac{1}{\rev{\eta_c}}\cdot \max \left\{\hat{p}^s_{\tau_j, i} + \delta_{\tau_{j -1}, i} - E^{\max}, 0 \right\}, & \quad t = \tau_j, \ \forall j\in [k],
\end{array}\right.\\
\tilde{p}^{ex}_{t,i} &= \left\{\begin{array}{ll} \hat{p}^{ex}_{t,i}, &\quad  \forall t \in \T \setminus \{\tau_0, \tau_1, \dots, \tau_k\}, \\
\hat{p}^{ex}_{t,i} - \min \{\hat{p}^d_{\tau_0,i},\; \hat{p}^{ex}_{\tau_0,i}\},  & \quad t = \tau_0,\\
\hat{p}^{ex}_{t,i} + \frac{1}{\rev{\eta_c}}\cdot \max \left\{\hat{p}^s_{\tau_j, i} + \delta_{\tau_{j -1}, i} - E^{\max}, 0 \right\}, &\quad  t = \tau_j, \ \forall j\in [k],
\end{array}\right.\\
\tilde{p}^{s}_{t,i} &= \left\{\begin{array}{ll} \hat{p}^s_{t,i}, &\quad  \forall t \in [\tau_0-1], \\
\hat{p}^{s}_{t,i} + \delta_{\tau_j,i}, &\quad  \forall t \in \{\tau_j, \dots, \tau_{j +1}-1\}, \; \forall j \in \{0, \dots, k\},\\
\end{array}\right.
\end{align*}
where $\tau_{k +1}-1 = T$.

The new solution $\tilde{\bm p}$ above is created in the following fashion: We first reduce both $\hat{p}^d_{\tau_0,i}$ and $\hat{p}^{ex}_{\tau_0, i}$, which causes $\hat{p}^s_{t, i}$ to increase in time periods following $\tau_0$. In particular, it may increase beyond $E^{\max}$. In order to fix this, we need to modify the charging levels (and excess values of corresponding time periods) to ensure that the storage levels meet the maximum requirement $E^{\max}$. We carefully decrease the values of $\hat{p}^{c}_{t,i}$ so that the minimum charge level is satisfied and the state-of-charge level is never beyond $E^{\max}$ at the same time.

\begin{restatable}{claim}{claim_pdpex_newp_feasibility}\label{claim:pdpex_newp_feasibility}
$(\tilde{\bm \theta}, \tilde{\bm f}, \tilde{\bm p}, \tilde{\bm u})$ is a feasible solution to \eqref{formulation:regularization} with given $\bm \lambda$.
\end{restatable}
\begin{proof}
It suffices to show that $\tilde{\bm p}$ satisfies $\eqref{eq:battery_pc_limits} - \eqref{eq:battery_storage}$, \eqref{eq:battery_ps_limits} and $\eqref{eq:battery_plspex_lb} - \eqref{eq:power_balance}$.

\begin{itemize}
    \item \eqref{eq:battery_pc_limits}: We want to show that $\tilde{p}^{c}_{t,i} \geq 0 = E^{\min}_c$ for all $t \in \T$. It is sufficient to prove this for $t = \{ \tau_1, \dots, \tau_k\}$. There are two cases:
    \begin{itemize}
        \item[(i)] If $\hat{p}^s_{\tau_j, i} + \delta_{\tau_{j -1}, i} - E^{\max} \leq 0$: This case is straightforward as $\tilde{p}^c_{t,i} = \hat{p}^c_{t,i} \geq 0.$
        \item[(ii)] If $\hat{p}^s_{\tau_j, i} + \delta_{\tau_{j -1}, i} - E^{\max} > 0$: In this case, note that 
        \rev{$\tilde{p}^c_{t,i} = \frac{1}{\eta_c} (\tilde{p}^s_{t,i} - \tilde{p}^s_{t-1,i}) > \frac{1}{\eta_c} (E^{\text{max}} - \tilde{p}^s_{t-1,i} -\delta_{\tau_{j-1},i}) \geq 0$, where the first equality comes from \eqref{eq:battery_storage}.}   
    \end{itemize}
    \item \eqref{eq:battery_pd_limits}: It is straightforward to verify that $\tilde{p}^d_{t,i} \geq 0 = E^{\min}_d$ for all $t \in \T$.
    \item \eqref{eq:battery_storage}: It is straightforward to verify that $\tilde{p}^s_{t,i} = \tilde{p}^s_{t-1,i} + \rev{\eta_c}\cdot \tilde{p}^c_{t,i} - {1}/{\rev{\eta_d}}\cdot\tilde{p}^d_{t,i}$ for all $t \in \T$.
    \item \eqref{eq:battery_ps_limits}: We want to show that $\tilde{p}^s_{t,i} \leq E^{\max}$ for all $t \in \T$. Clearly $\tilde{p}^s_{t,i} = \hat{p}^s_{t,i} \leq E^{\max}$ for all $t \in [\tau_0-1]$. For $t \geq \tau_0$, we show this in three parts: 
    \begin{itemize}
        \item[(i)] For $t = \tau_0$, from \eqref{eq:battery_storage} and the constructions above, we have $\hat{p}^s_{\tau_0, i} = \hat{p}^s_{\tau_0-1, i} - {1}/{\rev{\eta_d}}\cdot \hat{p}^d_{\tau_0, i} \leq E^{\max} - {1}/{\rev{\eta_d}} \cdot \min\{\hat{p}^d_{\tau_0,i},\; \hat{p}^{ex}_{\tau_0,i}\}.$ Therefore, $\tilde{p}^s_{\tau_0, i} = \hat{p}^s_{\tau_0, i} + {1}/{\rev{\eta_d}} \cdot \min\{\hat{p}^d_{\tau_0,i},\; \hat{p}^{ex}_{\tau_0,i}\} \leq E^{\max}.$
        \item[(ii)] For $t = \tau_j$ for all $j \in [k]$, we have $\tilde{p}^s_{\tau_j, i}  = \hat{p}^s_{\tau_j,i} + \delta_{\tau_j, i} = \hat{p}^s_{\tau_j,i}  + \delta_{\tau_{j -1}, i} - \max \{\hat{p}^s_{\tau_j, i} + \delta_{\tau_{j -1}, i} - E^{\max}, 0 \} \leq E^{\max}$. 
        \item[(iii)] Finally, for any $t \in \{\tau_j+1, \dots, \tau_{j +1}-1\}$ for $j \in \{0, \dots, k\}$, observe that since $\hat{p}^c_{\tau_j +1, i} = \dots = \hat{p}^c_{\tau_{j +1} - 1, i} = 0,$ we have that $\tilde{p}^s_{t, i} = \hat{p}^s_{t,i} + \delta_{\tau_j, i} \leq \hat{p}^s_{\tau_j, i} + \delta_{\tau_j, i} = \tilde{p}_{\tau_j, i} \leq E^{\max}$, where the last inequality follows from the above.
    \end{itemize}
    \item \eqref{eq:battery_plspex_lb}: It is straightforward to verify that $\tilde{p}^{ex}_{t,i} \geq 0$ for all $t \in \T$. 
    \item \eqref{eq:power_balance}: It is straightforward to verify that $\hat{p}^c_{t,i} + \hat{p}^{ex}_{t,i} - \hat{p}^{d}_{t,i} - \hat{p}^{ls}_{t,i} = \tilde{p}^c_{t,i} + \tilde{p}^{ex}_{t,i} - \tilde{p}^{d}_{t,i} - \tilde{p}^{ls}_{t,i}$ for all $t \in \T$.
\end{itemize}
\QEDA
\end{proof}
Observe that $\tilde{p}^d_{\tau_0,i}\cdot \tilde{p}^{ex}_{\tau_0,i} = 0.$
Next, we claim that the total additional adjustment made to $\delta$ over time is upper bounded.

\begin{restatable}{claim}{claim_pdpex_sum_ub}\label{claim:pdpex_sum_ub}
$\sum_{j = 1}^k\max\{\hat{p}^s_{\tau_j, i} + \delta_{\tau_{j -1}, i} - E^{\max}, 0 \} \leq |\delta_{\tau_0, i}|$.
\end{restatable}
\begin{proof}
Suppose that $S \subseteq [k]$ such that $\max\{\hat{p}^s_{\tau_j, i} + \delta_{\tau_{j -1}, i} - E^{\max}, 0 \} = \hat{p}^s_{\tau_j, i} + \delta_{\tau_{j -1}, i} - E^{\max}$. Let $S = \{j_1, j_2, \dots, j_m\}$. If $S = \emptyset$, then there is nothing to verify. Otherwise, it is straightforward to verify that:
\begin{align*}
    & \sum_{j = 1}^k\max\{\hat{p}^s_{\tau_j, i} + \delta_{\tau_{j}-1, i} - E^{\max}, 0 \} \\
    & = (\hat{p}^s_{\tau_{j_1}, i} + \delta_{\tau_{0}, i} - E^{\max}) + (\hat{p}^s_{\tau_{j_2}, i} + \delta_{\tau_{j_2}-1, i} - E^{\max}) + \cdots +(\hat{p}^s_{\tau_{j_m}, i} + \delta_{\tau_{j_m}-1, i} - E^{\max}) \\
    & = (\hat{p}^s_{\tau_{j_1}, i} + \delta_{\tau_{0}, i} - E^{\max}) + (\hat{p}^s_{\tau_{j_1}}, i - \hat{p}^s_{\tau_{j_2}, i}) + \cdots +(\hat{p}^s_{\tau_{j_{m-1}}, i} - \hat{p}^s_{\tau_{j_m} , i}) \\
    & =  \hat{p}^s_{\tau_{j_m}, i} + \delta_{\tau_{0}, i} - E^{\max} \leq |\delta_{\tau_{0}, i}|.
\end{align*}
The second equality is due to the fact that $\delta_{\tau_{j_l},i} = E^{\max} - \hat{p}^s_{\tau_{j_l},i}$ and $\delta_{\tau_{j_{l+1}}-1,i} = \delta_{\tau_{j_l},i}$ for all $l=1,...,m$.
\QEDA
\end{proof}
Finally, we are ready to compute the difference in objective function value of the two solutions:
\begin{align*}
     & c(\tilde{\bm p}) + \bm \lambda^\top g(\tilde{\bm p}) - c(\hat{\bm p}) - \bm \lambda^\top g(\hat{\bm p}) \\
     = & \sum_{t \in \T} \sum_{i \in N} \left( \tilde{p}^{ls}_{t,i} - \hat{p}^{ls}_{t,i} + \tilde{p}^{ex}_{t,i} - \hat{p}^{ex}_{t,i} + \lambda_c ( \tilde{p}^{c}_{t,i} - \hat{p}^{c}_{t,i}) + \lambda_d (\tilde{p}^{d}_{t,i} - \hat{p}^{d}_{t,i}) \right)\\
     = & - \min \{\hat{p}^d_{\tau_0,i},\; \hat{p}^{ex}_{\tau_0,i}\} + \sum_{j \in [k]} \frac{1}{\rev{\eta_c}}\cdot \max \left\{\hat{p}^s_{\tau_j, i} + \delta_{\tau_{j -1}, i} - E^{\max}, 0 \right\}\\
     & - \lambda_c \cdot \sum_{j \in [k]} \frac{1}{\rev{\eta_c}}\cdot \max \left\{\hat{p}^s_{\tau_j, i} + \delta_{\tau_{j -1}, i} - E^{\max}, 0 \right\} - \lambda_d \cdot \min \{\hat{p}^d_{\tau_0,i},\; \hat{p}^{ex}_{\tau_0,i}\}\\
     = & -\rev{\eta_d}|\delta_{\tau_0,i}| - \lambda_d \rev{\eta_d}|\delta_{\tau_0,i}| + \frac{1 - \lambda_c}{\rev{\eta_c}} \cdot \left(\sum_{j = 1}^k\max\left\{\hat{p}^s_{\tau_j, i} + \delta_{\tau_{j -1}, i} - E^{\max}, 0 \right\}\right) \\
          \le & -\rev{\eta_d}|\delta_{\tau_0,i}| - \lambda_d \rev{\eta_d}|\delta_{\tau_0,i}| + \frac{1 - \lambda_c}{\rev{\eta_c}} \cdot |\delta_{\tau_0,i}| \\
          = & \rev{\frac{|\delta_{\tau_0,i}|}{\eta_c} \left( 1 - \eta_c\eta_d - \lambda_c - \eta_c \eta_d \lambda_d \right) }.
     \end{align*}
     When \rev{$1 - \eta_c\eta_d - \lambda_c - \eta_c \eta_d \lambda_d < 0$} (i.e., under the assumption that \rev{$\lambda_c + \eta_c\eta_d \lambda_d > 1 - \eta_c\eta_d$}), $c(\tilde{\bm p}) + \bm \lambda^\top g(\tilde{\bm p}) < c(\hat{\bm p}) + \lambda^\top g(\hat{\bm p})$, hence a contradiction. 
\QEDA
\end{proof}

\section{Detailed Formulation for Trilevel \texorpdfstring{$N-k$}{Lg} Contingency Problem}
\label{sec:numerical_formulation}
\subsection{Problem Formulation}
We formally give a mathematical formulation of this problem. The first decision by a planner, denoted as $\bm x$, is to place batteries in the grid subject to the given budget $b \in \mathbb{Z}_+$.
\begin{subequations}
\label{eq:trilevel_x_y}
\begin{align}
    & \bm x \in \{0,\;1\}^\N, \label{eq:trilevel_x_y_a}\\
    & \sum_{i \in \N} x_i \le b.
\end{align}
Then a network interdictor decides whether or not to destroy a transmission line $(i,j) \in \L$ by a binary variable $y_{ij}$. The interruption can happen to at most $k$ transmission lines in the system and impacts the performance of transmission lines throughout the time period considered:
\begin{align}
    & \bm y \in \{0,\;1\}^\L, \\
    & \sum_{e \in \L} y_e \le k, \\
    & -F_{ij}(1-y_{ij}) \le f_{t,ij} \le F_{ij} (1-y_{ij}), && \forall  t \in \T, \; (i,j) \in \L. \label{eq:trilevel_x_y_e}
\end{align}
Initial state-of-charge as well as upper and lower bounds of state-of-charge of a battery depends on whether a battery is sited or not:
\begin{align}
    & p^{s}_{0,i} = E_{0}x_{i}, &&  \forall i \in \N, \label{eq:2stage_ps_initial} \\
    & E^{\min}x_{i} \le p^{s}_{t,i} \le E^{\max}x_{i}, && \forall  t \in \T, i \in \N.
\end{align}
When there is no battery at node $i \in \N$, $p^s_{t,i}=0$ for all $t \in \T$, so we can expand the state-of-charge over time to the entire $\N$:
\begin{align}
    p^{s}_{t,i} = p^{s}_{t-1,i} + \eta \cdot p^{c}_{t,i}- {1}/{\eta} \cdot p^{d}_{t,i}, && \forall  t \in \T, i \in \N \label{eq:2stage_ps}.
\end{align}
The bounds on charging and discharging, similarly, depend on whether a battery exists at a node and whether the battery is charging or discharging, represented by $u_{t,i}$:
\begin{align}
    & \bm u \in \{0,1\}^{T \times N}, \\
    & E_c^{\min} u_{t,i} \le p^{c}_{t,i} \le E_c^{\max} u_{t,i}, && \forall  t \in \T, i \in \N, \\
    & E_d^{\min} (x_{i}-u_{t,i}) \le p^{d}_{t,i} \le E_d^{\max} (x_{i}-u_{t,i}), && \forall  t \in \T, i \in \N \label{eq:2stage_pd_limits}.
\end{align}
Finally, a battery can only operate when there is a battery installed at the node:
\begin{align}
    & u_{t,i} \le x_i, && \forall  t \in \T, i \in \N. \label{eq:2stage_u_x}
\end{align}

\end{subequations}
 
Bounds on generator output \eqref{eq:generation_limits}, limits on transmission lines \eqref{eq:flow_limits}, nonnegativity constraint on load shedding and excess power \eqref{eq:battery_plspex_lb}, and power balance equation \eqref{eq:power_balance} do not change from the optimal power flow with the battery problem. Operational constraints for battery \eqref{eq:2stage_ps_initial} - \eqref{eq:2stage_u_x} are also the same as two-stage stochastic programming studied in the previous section. For purposes of this problem, we omit Ohm's law constraint (see, e.g., \citealt{johnson2022scalable}), and therefore the third-level problem becomes a network flow problem. Throughout the time period considered, the network operator then aims to generate power and send power flows to minimize the load shedding and lost power. We now provide the formulation below:
\begin{align}
\label{eq:trilevel_whole}
    \min_{\bm x} \max_{\bm y} \min_{\bm f,\bm p,\bm u} \left\{ c(\bm p)\colon \eqref{eq:generation_limits},\; \eqref{eq:flow_limits},\; \eqref{eq:battery_plspex_lb}, \;\eqref{eq:trilevel_x_y_a}\text{-}\eqref{eq:2stage_u_x}\right\}.
\end{align}

\subsection{Solution Methodology}
\label{sec:solution_methodology}
Computationally, a $\min$-$\max$-$\min$ problem is not an easy problem to solve. 
In literature, two-stage robust programs, a special case of the $\min$-$\max$-$\min$ problem, have been discussed extensively in the literature (see, e.g., \citealt{atamturk2007two,jiang2014two,van2017robust,mattia2017staffing}). 
\citet{Jeroslow1985} showed NP-hardness and \citet{ben-ayed1990} also discussed the computational difficulties of bilevel linear problems. 
Problem \eqref{eq:trilevel_whole} is a trilevel problem with binary variables at each level resulting in a particularly challenging optimization problem, which forbids the classical approach to formulating the trilevel problem into a bilevel problem and applying techniques to solve bilevel optimization problems.
Using the regularized MIP model, however, enables us to linearize the third-level problem and convert it to a bilevel problem by taking the dual of the third-level problem.

It is often undesirable to have unbounded dual variables. A popular heuristic is to use the big-$M$ method to bound such unbounded variables to a reasonable number. We prove that in our formulation the dual variables are bounded. This allows the exact reformulation for bilinear terms that appear in the objective function. 

Once we obtain the bilevel formulation, we apply a generic iterative algorithm to solve the bilevel optimization problem. In particular, we use the algorithm outlined in \citealt{bienstock2008computing}. The stopping criterion was either: (i) the iteration reaches maximum iteration of 1000; (ii) the iteration has run more than 6 hours; or (iii) the upper bound and lower bound gap is less than 0.5\%. 

\subsection{Detailed Bilevel Formulation}
We present here the reformulation to the bilevel $\min$-$\max$ problem.
Since variables associated with the third level are linear, we dualize the third level. For notational simplicity, let $\bm \theta = (\bm\alpha,\bm \beta^\pm, \bm \delta, \bm \gamma^\pm, \bm\tau,\bm \tau^0, \bm \mu^{\pm}, \bm\nu^\pm, \bm\omega^\pm, \bm \mu^\pm, \bm\phi)$. The primal variable associated with a constraint for the third level is provided on the left.
{
\allowdisplaybreaks
\begin{subequations}
\label{eq:trilevel_dual}
\begin{align}
\min_{\bm x} \max_{\bm y,\bm \theta, \bm z}\quad 
    & -\sum_{t \in \T} \sum_{i \in \N}D_{t,i} \alpha_{t,i}  -
     \sum_{t \in \T, (i,j) \in \L} F_{ij}(1-y_{ij}) (\beta^{+}_{t,ij} + \beta^{-}_{t,ij}) \span \span \notag \\
    & + \sum_{t \in \T} \sum_{i \in \N}G^{\min}_i \gamma^{+}_{t,i} - \sum_{t \in \T} \sum_{i \in \N}G^{\max}_i \gamma^{-}_{t,i} + \sum_{i \in \N} E^0 x_i \tau^0_{i} \span\span \notag \\
    & + \sum_{t \in \T} \sum_{i \in \N} E^{\min} x_{i} \mu^{+}_{t,i} - \sum_{t \in \T} \sum_{i \in \N} E^{\max} x_{i} \mu^{-}_{t,i} \notag \\
    & + \sum_{t \in \T} \sum_{i \in \N}E^{\min}_d x_{i} \omega^{+}_{t,i} - \sum_{t \in \T} \sum_{i \in \N}E^{\max}_d x_{i} \omega^{-}_{t,i} - \sum_{t \in \T} \sum_{i \in \N}x_i \phi_{t,i},  \span \span \label{eq:trilevel_dual_obj}\\
    \textup{s.t.} \quad \nonumber \\
    p^{ls}_{t,i},\;p^{ex}_{t,i}\quad \cdots\cdots\quad
    & -1 \le \alpha_{t,i} \le 1, && \forall t \in \T, i \in \N,  \label{eq:trilevel_dual_alpha_bounds} \\ 
    f_{t,ij} \quad \cdots\cdots\quad
    & \alpha_{t,i} - \alpha_{t,j} + \beta^+_{t,ij} - \beta^-_{t,ij} = 0,  && \forall t \in \T, (i,j) \in \L, \\ 
    p^{g}_{t,i} \quad \cdots\cdots\quad
    & -\alpha_{t,i} + \gamma^{+}_{t,i} - \gamma^{-}_{t,i} = 0,  && \forall t \in \T, i \in \N, \label{eq:trilevel_dual_alpha_gamma}\\ 
    p^{c}_{t,i} \quad \cdots\cdots\quad
    & \alpha_{t,i} - \eta \cdot \tau_{t,i} + \nu^{+}_{t,i} - \nu^{-}_{t,i} = \lambda^c, && \forall t \in \T, i \in \N, \label{eq:trilevel_dual_pc}\\ 
    p^{d}_{t,i} \quad \cdots\cdots\quad
    & -\alpha_{t,i} + {1}/{\eta} \cdot \tau_{t,i} + \omega^{+}_{t,i} - \omega^{-}_{t,i} = \lambda^d, && \forall t \in \T, i \in \N, \label{eq:trilevel_dual_pd}\\ 
    p^{s}_{t,i} \quad \cdots\cdots\quad
    & \tau_{t,i} - \tau_{t+1,i} + \mu^{+}_{t,i} - \mu^{-}_{t,i} = 0, && \forall  t \in \T\setminus\{T\}, i \in \N, \\ 
    p^{s}_{T,i} \quad \cdots\cdots\quad
    & \tau_{T,i} + \mu^{+}_{T,i} - \mu^{-}_{T,i} = 0, && \forall i \in \N,\\
    p^{s}_{0,i} \quad \cdots\cdots\quad
    & \tau^{0}_{i} -\tau_{1,i} = 0, && \forall i \in \N, \\ 
    u_{0,i} \quad \cdots\cdots\quad
    & E_c^{\min} \nu^{+}_{t,i} - E_c^{\max} \nu^{-}_{t,i} - E_d^{\min} \omega^{+}_{t,i} + E_d^{\max} \omega^{-}_{t,i}  + \phi_{t,i} \ge 0, && \forall t \in \T, i \in \N,  \label{eq:trilevel_dual_u}\\
    & \bm \beta^\pm, \bm \gamma^\pm, \bm\nu^\pm, \bm\omega^\pm,\bm \mu^\pm,\bm\phi \ge \bm 0, \\
    & \sum_{i \in \N} x_i \le b, \\
    & \sum_{l \in \L} y_l \le k, \\
    & \bm x \in \{0,1\}^{N},\; \bm y \in \{0,1\}^{L}.
\end{align}
\end{subequations}
}
\begin{restatable}{proposition}{propdualboundbeta}\label{prop_dual_bound_beta}
Independent of the values of $\lambda^c$ and $\lambda^d$, there exists an optimal solution of \eqref{eq:trilevel_dual} for which the following inequalities are valid: 
\begin{align*}
     & 0\leq  \beta^{\pm}_{t,ij} \leq  2, \quad \forall t \in \T, \; (i,j) \in \L.
\end{align*}
\end{restatable}

\begin{proof}
In the objective function \eqref{eq:trilevel_dual_obj}, we focus on optimizing $\bm \beta^+$ and $\bm \beta^-$. Notice that $F_{ij}\geq 0$ for all $(i,j) \in \L$ and  $t\in {T}$. Then, for a given $(i,j) \in \L$ and  $t\in {T}$ and the associated $\beta^+_{t,ij}, \beta^-_{t,ij}$, we optimize
\begin{align*}
\max_{\beta^+_{t,ij}\geq 0, \beta^-_{t,ij}\geq 0} &    -\beta^-_{t,ij} - \beta^+_{t,ij}, \\
    \textup{s.t.} \quad  &  \beta^+_{t,ij}  = \alpha_{t,j} -   \alpha_{t,i}  +\beta^-_{t,ij},
\end{align*}
which is equivalent to optimizing
\begin{align*}
v^\beta_{t,ij} =\min_{\beta^+_{t,ij}\geq 0, \beta^-_{t,ij}\geq 0} &    \beta^-_{t,ij} + \beta^+_{t,ij},\\
    \textup{s.t.} \quad  &  \beta^+_{t,ij} -\beta^-_{t,ij} = \alpha_{t,j} -   \alpha_{t,i},
\end{align*}
where the optimal value is $   v^{\beta^*}_{t,ij} =|\alpha_{t,j} - \alpha_{t,i}|$. From constraint \eqref{eq:trilevel_dual_alpha_bounds}, we know $-1\leq \alpha_{t,i}\leq 1$ and $-1\leq \alpha_{t,j}\leq 1$, then we have $0\leq v^{\beta^*}_{t,ij} \leq 2$, which implies that 
\begin{align*}
   0\leq \beta^-_{t,ij} + \beta^+_{t,ij} \leq 2.
\end{align*}
Hence, we have the desired result.
\QEDA
\end{proof}
We can then use McCormick inequalities to exactly reformulate the bilinear terms of the form $\bm \beta^{\pm} \bm y$ that appear in dualizing the third level of this trilevel problem.

In the objective function, the only bilinear terms are $\{y_{ij}\beta^+_{t,ij} \}_{t\in \T, (i,j)\in \L}$ and $\{y_{ij}\beta^-_{t,ij} \}_{t\in \T, (i,j)\in \L}$.
Since we show that all dual variables are bounded, especially $\bm 0 \le \bm \beta^\pm \le \bm 2$, bilinear terms can be reformulated exactly by applying the McCormick Envelopes (see, e.g., \citealt{mccormick1976computability}):
\begin{subequations}
\label{eq:bilinear_reform}
\begin{align*}
    & z^+_{t,ij} \geq 0, \quad z^+_{t,ij}  \geq \beta^+_{t,ij}+2y_{ij}-2, \quad  z^+_{t,ij} \leq \beta^+_{t,ij}, \quad z^+_{t,ij} \leq 2y_{ij}, && \forall  t \in \T, (i,j) \in \L,\\
     & z^-_{t,ij} \geq 0, \quad z^-_{t,ij}  \geq \beta^-_{t,ij}+2y_{ij}-2, \quad z^-_{t,ij} \leq \beta^-_{t,ij}, \quad z^-_{t,ij} \leq 2y_{ij}, && \forall t \in \T,(i,j) \in \L.
\end{align*} 
\end{subequations}
Note that this is an exact reformulation of the bilinear terms, not a relaxation. 
Hence, the objective function can replaced with the reformulation and additional constraints from \eqref{eq:bilinear_reform} are added to the formulation. This completes converting the trilevel formulation to bilevel formulation.

\end{appendices}

\end{document}